\documentclass{amsart}

\newtheorem{theorem}{Theorem}[section]
\newtheorem{proposition}[theorem]{Proposition}
\newtheorem{lemma}[theorem]{Lemma}
\newtheorem{corollary}[theorem]{Corollary}

\theoremstyle{definition}

\theoremstyle{remark}
\newtheorem{remark}[theorem]{Remark}

\numberwithin{equation}{section}

\begin{document}

\title{Non-isothermal non-Newtonian flow problem   with heat convection
and  Tresca's friction law}

%    Remove any unused author tags.

%    author one information
\author{Mahdi Boukrouche}
\address{Lyon University, F-42023 Saint-Etienne, Institut Camille Jordan CNRS UMR 5208, 23 Rue Paul Michelon, 42023 Saint-Etienne Cedex 2, France}
%\curraddr{}
\email{mahdi.boukrouche@univ-st-etienne.fr}
%\thanks{}

%  author two information 
\author{Hanene Debbiche}
\address{Lyon University, F-42023 Saint-Etienne, Institut Camille Jordan CNRS UMR 5208, 23 Rue Paul Michelon, 42023 Saint-Etienne Cedex 2, France}
\curraddr{University of Bordj bou Arr\'eridj, Algeria}
\email{hanane.debbiche@gmail.com}
%\thanks{}

%    author three information
\author{Laetitia Paoli}
\address{Lyon University, F-42023 Saint-Etienne, Institut Camille Jordan CNRS UMR 5208, 23 Rue Paul Michelon, 42023 Saint-Etienne Cedex 2, France}
%\curraddr{}
\email{laetitia.paoli@univ-st-etienne.fr}
%\thanks{}



%\date{30 mars 2021}

%\dedicatory{}

\begin{abstract}
%% Text of abstract
We consider an incompressible non-isothermal fluid flow with non-linear slip boundary conditions governed by Tresca's friction law. We assume that the stress tensor is given as $\sigma = 2 \mu\bigl( \theta, u, | D(u) |) |D(u) |^{p-2} D(u) -  \pi {\rm Id}$ where $\theta$ is the temperature, $\pi$ is the pressure, $u$ is the velocity and $D(u)$ is the strain rate tensor of the fluid while $p$ is a real parameter. The problem is thus given by the $p$-Laplacian Stokes system with subdifferential type boundary conditions coupled to a $L^1$ elliptic equation describing the heat conduction in the fluid. We establish first an existence result for a family of approximate coupled problems where the $L^1$ coupling term in the heat equation is replaced by a bounded one depending on a parameter $0<\delta <<1$, by using a fixed point technique. Then we pass to the limit as $\delta$ tends to zero and we prove the existence of a solution $(u, \pi, \theta)$ to our original coupled problem in Banach spaces depending on $p$ for any $p > 3/2$.
\end{abstract}

\maketitle

%% \linenumbers

%% main text

\section{Introduction} \label{section1}

Many industrial applications, like lubrication or extrusion / injection problems, involve complex fluids like molten polymers, oils, or colloidal fluids which do not satisfy the usual linear Newton's relationship between the stress tensor,  the strain rate tensor and the pressure and obey rather a power law of the form
\begin{eqnarray*}
\sigma = 2 \mu \bigl( \theta, u, \bigl| D(u) \bigr|) \bigl|D(u) \bigr|^{p-2} D(u) -  \pi {\rm Id}
\end{eqnarray*}
 where $\theta$ is the temperature, $\pi$ is the pressure, $u$ is the velocity and $D(u)$ is the strain rate tensor of the fluid while $p$ is a real parameter. With $p>2$ we obtain a description of dilatant (or shear thickening) fluids like colloidal fluids and for $p \in (1,2)$ we recover pseudo-plastic (or shear thinning) fluids like  molten polymers (see \cite{barnes, mewis, wagner} for instance).  It follows that both the velocity and the pressure of the fluid will belong to Banach spaces depending on $p$. For a constant function $\mu$ and classical boundary conditions like Dirichlet, Neumann or Navier boundary conditions several existence, uniqueness and regularity results have been already established (see  \cite{Lions69, malek} and the references therein or more recently \cite{berselli, diening} for unsteady flows).
 
When $p=2$ we obtain a generalization of Navier-Stokes Newtonian fluids since the viscosity still depends on the temperature, the velocity and the modulus of the strain rate tensor via  an appropriate choice of the function $\mu$ and it allows to consider fluids like oils (see \cite{ladi} for instance).

\smallskip
 
Nevertheless several experimental studies have shown that  such flows exhibit also non-standard  behavior at the boundary with non-linear slip phenomena of friction type (see for instance   \cite{magnin, Hervet2003}).
The first mathematical results  with such boundary conditions have  been proposed by H.Fujita et al in \cite{fujita1, F1, F2, F3, saito1, F6, saito2} for stationnary Stokes flows and developped later on for  Navier-Stokes flows when $p=2$ (\cite{Roux2, Roux3}). 

\smallskip

Since friction generates heat,  thermal effects can not be neglected and   the energy conservation law leads to  a elliptic equation in $L^1$ for the temperature. Moreover,  the convection term yields to some compatibility condition between the regularity of the temperature and velocity of the fluid which both will belong to Banach spaces depending on the parameter $p$.

\smallskip

The rest of this paper is devoted to the study of this highly non-linear  coupled fluid flow / heat transfer problem and to the proof of an existence result.
In Section \ref{section2} we introduce the functional framework and 
starting from the conservation laws of mass, momentum and energy, we derive the mathematical formulation of the problem as a variational inequality describing the fluid flow coupled to an elliptic equation in $L^1$ describing  heat conduction into the fluid. In Section \ref{section3} we study an auxiliary flow problem where the two first arguments of the viscosity mapping are considered as given parameters. Next in Section \ref{section4} we introduce an approximate linearized heat equation depending on a parameter $\delta$, with $0<\delta <<1$, reminiscent of the technique proposed in \cite{Bocca2} for elliptic equations with $L^1$ right-hand side. We establish uniform estimates of the solutions to the auxiliary decoupled  fluid flow / heat transfer problems in appropriate Banach spaces. Then in Section \ref{section5} we prove the existence of a solution to an approximate coupled problem depending on the parameter $\delta$ by using a fixed point technique and finally we pass to the limit as $\delta$ tends to zero to obtain a solution $(u, \pi, \theta)$ to our original coupled problem.

%Then we study in Section \ref{section3} the fluid flow problem for a given temperature and in Section \ref{section4} the heat transfer equation for a given fluid velocity. Finally we apply an iterative successive approximation technique to establish the existence of a solution $(u, \pi, \theta)$ to the coupled problem.

\bigskip

%%%%%%%%%%%%%%%%%%%%%%%%%%%%%%%%%%%%%%%%%%%%%%%%%%%%%%%%%%%%%%%%%%%%%%%%%%%%%%%%%%%%%%%%%%%
%                           Debut de la section 2
%%%%%%%%%%%%%%%%%%%%%%%%%%%%%%%%%%%%%%%%%%%%%%%%%%%%%%%%%%%%%%%%%%%%%%%%%%%%%%%%%%%%%%%%%%%  

% \renewcommand{\theequation}{2.\arabic{equation}}
% \setcounter{equation}{0}
\section{Description of the problem} \label{section2}

Motivated by lubrication and extrusion / injection phenomena we consider an incompressible  fluid flow in the domain
 \begin{eqnarray*}
 \Omega=\Bigl\{(x' , x_{3}) \in \mathbb{R}^{2}\times \mathbb{R} :  \, x' \in \omega, \,  0< x_{3} < h(x')\Bigr\},
 \end{eqnarray*} 
 where $\omega$ is a non empty  bounded  domain of $\mathbb{R}^{2}$ with a Lipschitz continuous boundary 
 and $h$ is a Lipschitz continuous function  which is bounded from above and from below by some positive real numbers.

 In the stationnary case the conservation of mass and momentum leads to the system
 \begin{eqnarray} \label{NS}
\left\{ \begin{array}{ll}
 - {\rm div} (\sigma) = f \quad \hbox{\rm in $\Omega$} \\
 {\rm div} (u) = 0 \quad \hbox{\rm in $\Omega$}
 \end{array} \right.
 \end{eqnarray}
 where $u$ is the velocity of the fluid, $\sigma$ is the stress tensor and $f$ represents the vector of external forces. In order to be able to consider a broad class of fluids like oils, polymers or colloidal fluids we introduce the following   very general constitutive law depending on a real parameter $p \in (1, + \infty)$
 \begin{eqnarray} \label{sigma}
 \sigma = 2 \mu \bigl( \theta, u, \bigl| D(u) \bigr| \bigr) \bigl|D(u) \bigr|^{p-2} D(u) - \pi {\rm Id}_{{\mathbb{R}^3}}
 \end{eqnarray}
 where $\mu$ is a given function, $\theta$ is the temperature, $\pi$ is the pressure and $D(u) = \bigl(d_{ij}(u) \bigr)_{1\le i,j \le 3}$ is the strain rate tensor given by
 \begin{eqnarray*}
 d_{ij}(u) =\frac{1}{2} \left( \frac{\partial u_{i}}{\partial x_{j}} +\frac{\partial u_{j}}{\partial   x_{i}}\right) 
\quad  1 \le i,j \le 3.
\end{eqnarray*}
With $p>2$ we recover dilatant (or shear thickening) fluids like colloidal fluids and $p \in (1,2)$ corresponds to pseudo-plastic (or shear thinning) fluids like molten polymers.

We assume that the fluid is subjected to non-homogeneous Dirichlet boundary conditions on a part of $\partial \Omega$ and to non-linear slip boundary conditions of friction type on the other part. More precisely we decompose $\partial \Omega$ as  $\partial\Omega = \Gamma_{0}\cup\Gamma_{L}\cup\Gamma_{1}$ 
 with 
 \begin{eqnarray*}
 \Gamma_{0} =\Bigl\{(x' , x_{3}) \in \overline{\Omega} :  \, x_{3} =0 \Bigr\}, \quad 
  \Gamma_{1}=\Bigl\{(x' , x_{3}) \in \overline{\Omega} :  \, x_{3} = h(x') \Bigr\}
 \end{eqnarray*} 
and   $\Gamma_{L}$ is the lateral part of $\partial \Omega$. 

We introduce  a function  $g : \partial\Omega \to \mathbb{R}^{3}$  such that
 \begin{eqnarray}\label{sdf}
\int_{\Gamma_{L}} g\cdot n\,dY = 0,\quad g=0 \mbox{ on }\Gamma_{1}, \quad g\cdot n=0 \mbox{  on  } \Gamma_{0}, 
 \quad g\neq0  \mbox{ on } \Gamma_{L},         
\end{eqnarray}
where  $n= (n_{1} , n_{2} , n_{3} )$  is the unit outward normal vector to   $\partial\Omega$.
We denote here $v\cdot w$ the Euclidian inner product in  ${\mathbb R}^3$ of two vectors $v$ 
and $w$ and by $|\cdot|$ the Euclidian norm. 
We define by $u_{n}= u\cdot n$ and  $u_{\tau} = u- u_{n} n $ 
the normal and the tangential velocities on $\partial\Omega$. The  normal and tangential components  of the stress vector on $\partial \Omega$ are given by $\sigma_n$ and $\sigma_{\tau}$ with 
\begin{eqnarray*}
\displaystyle \sigma_n = \sum_{i,j=1}^3 \sigma_{ij} n_{j}n_{i}, \quad
\displaystyle \sigma_{\tau} =\left( \sum_{j=1}^3 \sigma_{ij} n_{j} -\sigma_{n}n_{i} \right)_{1 \le i \le 3}.
\end{eqnarray*}

As usual in lubrication or extrusion / injection problems the upper part of the boundary is fixed while the lower part is moving with a velocity $s: \Gamma_0 \to {\mathbb R}^2 \times \{0 \}$. So we have 
\begin{eqnarray}\label{7}
u=0  \quad \mbox{on}\    \Gamma_{1}, \quad u=  g \quad\mbox{on}\   \Gamma_{L}
\end{eqnarray}
and the slip condition 
\begin{equation}\label{8}
u_{n}=0 \quad \mbox{on} \ \Gamma_{0}
\end{equation}
combined with a subdifferential condition for the relative velocity $u_{\tau} -s$ on $\Gamma_0$ given by 
\begin{eqnarray*} \label{8bis}
u_{\tau} - s \in - \partial \psi_{{\overline B}_{{\mathbb R}^3} (0, k)} (\sigma_{\tau} )
\end{eqnarray*}
where $k: \Gamma_0 \to {\mathbb R}_*^+$ is the friction threshold, $\psi_{{\overline B}_{{\mathbb R}^3} (0, k)}$ is the indicatrix function of the closed ball ${\overline B}_{{\mathbb R}^3} (0, k)$ and $\partial \psi_{{\overline B}_{{\mathbb R}^3} (0, k)}$ its subdifferential (see \cite{rock}). This inclusion  means that the shear stress $\sigma_{\tau}$ remains bounded by $k$ and  slip occurs only when the threshold $k$ is reached, with a relative velocity in the opposite direction to $\sigma_{\tau}$ i.e. $u_{\tau}$ is unknown and satisfies Tresca's friction law \cite{Duvaut-Lions72}:
\begin{equation}\label{tr}
\left .\begin{array}{l}
|\sigma_{\tau}|<k\Rightarrow  u_{\tau}=s  \\
 |\sigma_{\tau}|=k\Rightarrow\exists\lambda\geq0\quad u_{\tau}= s-\lambda 
\sigma_{\tau}\\ 
\end{array}\right\}\quad \mbox{on} \ \Gamma_{0}.
\end{equation} 

Since friction generates heat we have to take also into account thermal effects in the fluid. By using Fourier's law for the heat flux we get the following heat equation 
\begin{eqnarray}\label{TEM1}
u\cdot \nabla\theta - div(K \nabla \theta) 
= 2\mu \bigl(\theta,u , \bigl|D(u)  \bigr| \bigr) \bigl|D(u) \bigr|^p +r(\theta) \quad \mbox{in} \ \Omega 
\end{eqnarray}
where $K$ is the thermal conductivity tensor and $r$ is a real function. We assume mixed Dirichlet-Neumann boundary conditions on $\Gamma_1 \cup \Gamma_L$ and $\Gamma_0$ i.e.
\begin{eqnarray}\label{TEM3}%diri}
\theta=0\quad \mbox{on}\  \Gamma_{1}\cup \Gamma_{L},  \quad 
( K\nabla \theta) \cdot n  =\theta^{b} \quad \mbox{on}\  \Gamma_{0}
\end{eqnarray}
where   $\theta^{b} $  is a given heat flux on  $\Gamma_{0}$.

\bigskip

Let us introduce now the functional framework. We   denote here and throughout this paper  by $ \textbf{X} $  the functional space $ X^3 $ and we define 
\begin{eqnarray*}
W^{1, q}_{ \Gamma_{1}\cup \Gamma_{L}}(\Omega)
=\Bigl\{\varphi \in W^{1, q}(\Omega)  : \  \varphi=0  \mbox{ on}\,\Gamma_{1}\cup \Gamma_{L} \Bigr\} \quad \forall q \ge 1
\end{eqnarray*}
endowed with the norm
\begin{eqnarray*}
\Vert \varphi \Vert_{1.q} =\left( \int_{\Omega} \bigl| \nabla \varphi \bigr|^q \, dx \right)^{1/q} \quad \forall \varphi \in W^{1, q}_{ \Gamma_{1}\cup \Gamma_{L}}(\Omega)
\end{eqnarray*}
and for all $p>1$
\begin{eqnarray*}
V_{0}^{p}=\Big\{\varphi\in \textbf {W}^{1,p}_{\Gamma_{1}\cup 
\Gamma_{L}}(\Omega)  : \ 
    \varphi\cdot n=0 \, \mbox{ on }\,\Gamma_{0}\Big\},
 \quad 
    V_{0.div}^{p}=\Bigl\{\varphi\in V_{0}^{p} : \  
{\rm div} (\varphi)=0  \,\mbox{ in  }\,\Omega \Bigr\} .
%\quad \forall p>1.
\end{eqnarray*}
%  \quad 
%W_{0.div}^{1,p}=\{\varphi\in\textbf{W}_{0}^{1,p}(\Omega);
%\quad div(\varphi)=0\quad \mbox{in}\quad\Omega\},
%$$
%$$
%V^p_{\Gamma_{1}}=\{\varphi\in \mathbf{W}^{1,p}(\Omega);\quad \varphi=0\quad\mbox{on}\quad \Gamma_{1}\},
%\qquad
%L_{0}^{p'}(\Omega)=\Big\{  u\in L^{p'}(\Omega);\quad\int_{\Omega} u(x)\,dx=0\Big\},
Let $p>1$. We assume that $g$ admits an extension $G$ to $\Omega$ such that
 \begin{eqnarray}\label{eqG} 
 G\in \textbf{W}^{1,p}(\Omega) ,  \quad    
{\rm div } (G)=0  \mbox{ in } \Omega,\quad G=0 \mbox{ on } \Gamma_{1},\quad G\cdot n=0 \mbox{ on } \Gamma_{0} 
  \end{eqnarray}  
 and that the data $f$, $s$ and $k$ satisfy
 \begin{eqnarray} \label{eqfk}
 f \in {\bf L}^{p'} (\Omega), \quad s \in {\bf L}^p(\Gamma_0), \quad k \in L^{p'}_+ (\Gamma_0) 
 \end{eqnarray}
 where $\displaystyle p' = \frac{p}{p-1}$ is the conjugate number of $p$. Moreover we assume that the mapping $\mu$ satisfies:
 \begin{eqnarray}\label{rop}
  (o,e,d)\mapsto\mu (o,e,d)\quad \mbox{is continuous on } 
  \mathbb{R}\times\mathbb{R}^{3}\times \mathbb{R}_{+},
\end{eqnarray}
 \begin{eqnarray}\label{m5}
 d\mapsto \mu(.,.,d) \quad \mbox{is monotone  increasing   on } 
 \mathbb{R_{+}},
\end{eqnarray} 
\begin{eqnarray}\label{mlo}
\hbox{\rm there exists  $(\mu_0, \mu_1) \in {\mathbb R}^2$ s.t. } \  0<\mu_{0}\leq\mu(o,e,d)\leq\mu_{1}
\quad \hbox{\rm for all  $(o,e,d)\in \mathbb{R}\times\mathbb{R}^{3}\times\mathbb{R}_{+}$}.
\end{eqnarray}
 For all $\theta \in L^q(\Omega)$ with $q \ge 1$ we let
 \begin{eqnarray*}
 a(\theta;  \upsilon , \varphi)=\int _{\Omega} \mathcal{F} \bigl(\theta, \upsilon +G , D(\upsilon +G) \bigr)
: D(\varphi)dx \quad \forall (\upsilon, \varphi) \in \bigl( V_{0}^{p} \bigr)^2
 \end{eqnarray*}
 where  $\mathcal{F}:  \mathbb{R}\times\mathbb{R}^{3}\times\mathbb{R}^{3\times3}\rightarrow \mathbb{R}^{3\times3}$  
is defined by 
\begin{eqnarray}\label{F}
\left\{ \begin{array}{ll}
\displaystyle \mathcal{F}(\lambda_{0}, \lambda_{1},\lambda_{2})= 2\mu(\lambda_{0} ,\lambda_{1},|\lambda_{2}
|)|\lambda_{2}|^{p-2}\lambda_{2} \quad \hbox{\rm if }  \lambda_2 \not = 0_{\mathbb{R}^{3\times3}},  \\
\displaystyle 
%\quad 
\mathcal{F}(\lambda_{0}, \lambda_{1},\lambda_{2})= 0_{\mathbb{R}^{3\times3}} \quad {\rm otherwise.}
\end{array} \right.
\end{eqnarray}
With  (\ref{mlo}) we have immediately 
\begin{eqnarray*}\label{F1}
  |\mathcal{F}(\lambda_{0}, \lambda_{1},\lambda_{2})|\leq 2\mu_{1}|\lambda_{2}|^{p-1}  
  \quad \forall (\lambda_{0}, \lambda_{1},\lambda_{2})\in 
  \mathbb{R}\times\mathbb{R}^{3}\times\mathbb{R}^{3\times3}.
\end{eqnarray*}
So $\mathcal{F}$ is continuous on $\mathbb{R}\times\mathbb{R}^{3}\times\mathbb{R}^{3\times3}$ and the mapping $a (\theta; \cdot, \cdot)$ is well defined and continuous on $V_{0}^{p} \times V_{0}^{p}$ for all $p>1$.

Then we introduce the new unknown function $\upsilon = u - G$ in order to deal with homogeneous boundary conditions on $\Gamma_{1}\cup\Gamma_{L}$ and  the flow problem (\ref{NS})-(\ref{sigma}) with  the boundary conditions (\ref{7})-(\ref{8})-(\ref{tr}) admits the following variational formulation
(\cite{Duvaut-Lions72}):
 
 \smallskip
 
\noindent  Find  $ \upsilon\in V_{0.div}^{p}$ and  $\pi \in 
L_{0}^{p'}(\Omega)$ such that
\begin{eqnarray*}\label{jh}
a(\theta;\upsilon , \varphi-\upsilon)- \int_{\Omega} \pi {\rm div} (\varphi) \, dx 
+\Psi(\varphi)-\Psi(\upsilon)\geq \int_{\Omega} f\cdot
(\varphi-\upsilon)\,dx 
\quad \forall\varphi\in V_{0}^{p}
\end{eqnarray*} 
with 
\begin{eqnarray*}
 \Psi(\varphi)=\int _{\Gamma_{0}}k|\varphi+G-s|\,dx' \quad  \forall\varphi\in V_{0}^{p}.
\end{eqnarray*}

\bigskip

Next we observe that the first term in the right hand side of the heat equation (\ref{TEM1}) belongs to $L^1(\Omega)$ whenever $u = \upsilon +G \in {\bf W}^{1,p}_{\Gamma_1} (\Omega)$. It follows that we may expect only $\theta \in W^{1,q}_{\Gamma_1 \cup \Gamma_L}(\Omega)$ with $1 \le q < 3/2$ and the convection term $u \cdot \nabla \theta$ in the left hand side of (\ref{TEM1}) belongs to $L^1(\Omega)$ only if $u \in {\bf L}^{p_*} (\Omega)$ with $p_* \ge q'$ where $q'$ is the conjugate number of $q$. Since ${\bf W}^{1,p}_{\Gamma_1} (\Omega) \subset {\bf L}^{p_*} (\Omega)$ with $\displaystyle p_* =  \frac{3p}{3-p}$ if $3>p>1$, any real $p_*$ if $p=3$ and $p_* = + \infty$ if $p >3$, we get $\displaystyle \frac{3}{2} > q \ge \frac{3p}{4p - 3}$ if $p<3$ which implies that $p$ should be greater that $\displaystyle \frac{3}{2}$. Thus we obtain a compatibility condition on $p$ and $q$ and we look for $\theta \in W^{1,q}_{\Gamma_1 \cup \Gamma_L}(\Omega)$ with 
\begin{eqnarray} \label{compa1}
%\left\{ 
\begin{array}{ll}
\displaystyle q \in \left[ 1, \frac{3}{2} \right) \quad \hbox{\rm if $p>3$,} \quad 
\displaystyle q \in \left( 1, \frac{3}{2} \right) \quad \hbox{\rm if $p=3$} \quad \\
\displaystyle \hbox{\rm and } \quad q \in \left[ \frac{3p}{4p -3} , \frac{3}{2} \right) \quad \hbox{\rm if $\displaystyle \  \frac{3}{2}<p<3$}
\end{array}
%\right.
\end{eqnarray}
or equivalently
\begin{eqnarray} \label{compa2}
p >3 \quad  \hbox{\rm if $q=1$ and} \quad  p \ge \frac{3q}{4q -3}  \quad  \hbox{\rm if $\displaystyle q \in \left( 1, \frac{3}{2} \right)$.}
\end{eqnarray}

\begin{remark} We may observe that $\displaystyle t \mapsto \frac{3t}{4t-3}$ is a decreasing function of $t$ on $[1, + \infty)$.
%varies from $3$ to $\displaystyle \frac{3}{2}$ when $t$ varies from $1$ to $\displaystyle \frac{3}{2}$. 
So, as  expected intuitively,  the more regular the temperature, the less regular the fluid velocity.
%and the minimal regularity property required for the fluid velocity is $\upsilon + G \in {\bf W}^{1, 3/2}_{\Gamma_1} (\Omega)$.
\end{remark}

As usual we assume that the thermal conductivity tensor satisfies
   \begin{eqnarray} \label{TEM2}
   K \in \bigl( L^{\infty} (\Omega) \bigr)^{3 \times 3}
   \end{eqnarray}
and
\begin{eqnarray} \label{TEM2bis}
\begin{array}{ll}
\displaystyle    \hbox{\rm there exists $ k_{0} >0$ s.t. } \\
\displaystyle   \sum_{i,j=1}^3 K_{i j} (x) \xi_i \xi_j \ge  k_{0} \sum_{i=1}^{3} |\xi_i|^2 \quad \hbox{\rm for all} \  \xi \in {\mathbb R}^3, \quad  \hbox{\rm for a.e.} \ x \in \Omega.
\end{array}
  \end{eqnarray}
  We assume also that
  \begin{eqnarray} \label{Cr}
  \hbox{\rm the mapping $r : {\mathbb R} \to {\mathbb R}$ is  continuous and uniformly bounded}
  \end{eqnarray}
  and 
  \begin{eqnarray} \label{thetab}
  \theta^b \in L^1 (\Gamma_0).
  \end{eqnarray}
Then we can derive a weak formulation of the heat equation (\ref{TEM1}) with the boundary conditions (\ref{TEM3}) by choosing a test-function  $w \in C^1 ({\overline  \Omega})$ such that $w=0$ on $\Gamma_1 \cup \Gamma_L$ (see for instance \cite{Boccardo-Neumann-boundary}).

\bigskip

Finally for any $p>3/2$ and $\displaystyle q \in \bigl[1, 3/2 \bigr)$ satisfying (\ref{compa1}) the weak formulation of the coupled heat-flow problem (\ref{NS})-(\ref{sigma})-(\ref{7})-(\ref{8})-(\ref{tr}) and (\ref{TEM1})-(\ref{TEM3}) is given by

 \bigskip
 
\noindent{\bf Problem (P)}   $ \quad $ Find  $ \upsilon\in V_{0.div}^{p}$,   $\pi \in 
L_{0}^{p'}(\Omega)$ and $\theta \in W^{1, q}_{ \Gamma_{1}\cup \Gamma_{L}}(\Omega)$ such that
\begin{eqnarray*}\label{jh}
a(\theta;\upsilon , \varphi-\upsilon)- \int_{\Omega} \pi {\rm div}(\varphi) \, dx
+\Psi(\varphi)-\Psi(\upsilon)\geq \int_{\Omega} f \cdot
(\varphi-\upsilon)\,dx,
\quad \forall\varphi\in V_{0}^{p}
\end{eqnarray*} 
and 
\begin{eqnarray*}
\begin{array}{ll}
\displaystyle  \int _{\Omega} \bigl((\upsilon+G) \cdot\nabla\theta \bigr)w\,dx
+ \int _{\Omega} (K \nabla \theta) \cdot \nabla w\,dx  \\
\displaystyle 
= \int _{\Omega}2 \mu \bigl(\theta , \upsilon +G, \bigl|D(\upsilon +G) \bigr| \bigr)
\bigl|D(\upsilon +G) \bigr|^p w \,dx 
 \\
\displaystyle 
+\int _{\Omega} r(\theta ) w \,dx + \int_{\Gamma_{0}} \theta^{b} w \,dx' \quad \forall w \in C^1 ({\overline  \Omega}) \ {\rm s.t.}  \ w = 0 \ {\rm on} \ \Gamma_1 \cup \Gamma_L.
\end{array}
\end{eqnarray*}

%%%%%%%%%%%%%%%%%%%%%%%%%%%%%%%%%%%%%%%%%%%%%%%%%%%%%%%%%%%%%
%                Debut Section 3
%%%%%%%%%%%%%%%%%%%%%%%%%%%%%%%%%%%%%%%%%%%%%%%%%%%%%%%%%%%%%

% \renewcommand{\theequation}{3.\arabic{equation}}
% \setcounter{equation}{0}
\section{Fluid flow auxiliary problem} \label{section3}

In this section we consider the following auxiliary flow problem
\begin{eqnarray} \label{flow_aux}
\left\{
\begin{array}{ll}
\displaystyle \hbox{\rm Find $\tilde \upsilon \in V_{0.div}^{p}$ such that} \\
\displaystyle 
\bigl\langle \mathbf{A}_{\theta, \upsilon} (\tilde \upsilon) , {\varphi} - \tilde \upsilon \bigr\rangle
+\Psi(\varphi)-\Psi(\tilde \upsilon) 
\geq \int_{\Omega} f\cdot (\varphi- \tilde \upsilon)\,dx 
\quad \forall \varphi\in V_{0.div}^{p}
\end{array}
\right.
\end{eqnarray}
where  $\langle \cdot, \cdot \rangle$ denotes the duality bracket between $ V_{0.div}^{p}$ and $\bigl(  V_{0.div}^{p} \bigr)'$ and  $\mathbf{A}_{\theta, \upsilon} :V_{0.div}^{p}\rightarrow(V_{0.div}^{p})'$ is the operator defined by  
 \begin{eqnarray*}
 \bigl\langle\mathbf{A}_{\theta, \upsilon} ( \tilde \upsilon) , {\varphi} \bigr\rangle 
 = \int_{\Omega} {\mathcal F} \bigl( \theta, \upsilon + G, D( \tilde \upsilon  +G) \bigr) : D(\varphi) \, dx \quad \forall (\tilde \upsilon , \varphi) \in   V_{0.div}^{p} \times V_{0.div}^{p}
 \end{eqnarray*}
for a given temperature $\theta \in L^q(\Omega)$ and a given velocity $\upsilon \in {\bf L}^p(\Omega)$, with $q \ge 1$ and $p>1$.

\begin{lemma} \label{lemma1}
Let $p>1$. Assume that   (\ref{m5})  and (\ref{mlo})     hold. Then the mapping  
$\lambda_{2}\mapsto\mathcal{F}(\cdot, \cdot ,\lambda_{2})$   
 is   monotone on $\mathbb{R}^{3\times3}$.
\end{lemma}

\begin{proof}
 Let us  prove now  that
\begin{eqnarray*}
\begin{array}{ll}
\displaystyle  \bigl(\mathcal{F}(\lambda_{0},\lambda_{1},\lambda_{2})-\mathcal{F}(\lambda_{0},\lambda_{1},\lambda_{2}
') \bigr):(\lambda_{2}-\lambda_{2}')\geq 0  \\
\displaystyle  
\forall (\lambda_{2},\lambda_{2}') \in (\mathbb{R}^{3\times 3})^2, \quad \forall (\lambda_0, \lambda_1) \in \mathbb{R} \times \mathbb{R}^3 .
\end{array}
\end{eqnarray*}

With (\ref{mlo}) the result is obvious if 
$\lambda_2 = 0_{\mathbb{R}^{3\times3}}$ and/or $\lambda'_2 = 0_{\mathbb{R}^{3\times3}}$.
So let us assume from now on that $\lambda_2 \not= 0_{\mathbb{R}^{3\times3}}$
and $\lambda'_{2} \not= 0_{\mathbb{R}^{3\times3}}$. 
We distinguish two cases.

\smallskip

\noindent {\bf Case 1:} $p \ge 2$.

\smallskip

Let 
\begin{eqnarray*}\label{expo}
F_{1}=\mu(\lambda_{0},\lambda_{1},|\lambda_{2}|)|\lambda_{2}|^{p-2},
\quad
 F_{2}=\mu(\lambda_{0},\lambda_{1},|\lambda_{2}'|)|\lambda_{2}'|^{p-2}.
 \end{eqnarray*}
 We have
\begin{eqnarray*}
\bigl( \mathcal{F}(\lambda_{0},\lambda_{1},\lambda_{2})-\mathcal{F}(\lambda_{0},\lambda_{1},\lambda_{2}
') \bigr):(\lambda_{2} -\lambda_{2}')   
= 2F_{1}|\lambda_{2}|^2-2 (F_{1} + F_2) \lambda_{2}:\lambda_{2}' +2F_{2}|\lambda_{2}'|^2.
\end{eqnarray*}
Observing that $F_1$ and $F_2$ are non-negative and  
$ 2 \lambda_{2}:\lambda_{2}'\leq |\lambda_{2}|^2 + |\lambda_{2}'|^2$, 
we get
\begin{eqnarray*}
&& \bigl( \mathcal{F}(\lambda_{0},\lambda_{1},\lambda_{2})-\mathcal{F}(\lambda_{0},\lambda_{1},\lambda_{2}
') \bigr):(\lambda_{2}-\lambda_{2}') \\
 && \geq 
2F_{1}|\lambda_{2}|^2-(F_{1}+F_{2}) \bigl( |\lambda_{2}|^2+|\lambda_{2}'|^2 \bigr)+2F_{2}
|\lambda_{2}'|^2
= 
(F_{1}-F_{2}) \bigl(|\lambda_{2}|^2-|\lambda_{2}'|^2 \bigr).
\end{eqnarray*}
and by replacing $F_1$ and $F_2$
\begin{eqnarray*}
&& \bigl( \mathcal{F}(\lambda_{0},\lambda_{1},\lambda_{2})-\mathcal{F}(\lambda_{0},\lambda_{1},\lambda_{2}
') \bigr):(\lambda_{2}-\lambda_{2}') \\
&&
\geq
 \bigl( \mu(\lambda_{0},\lambda_{1},
|\lambda_{2}|)|\lambda_{2}|^{p-2} -\mu(\lambda_{0},\lambda_{1},|\lambda_{2}'|)
|\lambda_{2}'|^{p-2} \bigr) \bigl(|\lambda_{2}|^2-|\lambda_{2}'|^2 \bigr)\\
&&=
 |\lambda_{2}|^{p-2} \big(\mu(\lambda_{0},\lambda_{1},|\lambda_{2}|)
 -\mu(\lambda_{0},\lambda_{1},|\lambda_{2}'|)\big)
 \big(|\lambda_{2}|^2-|\lambda_{2}'|^2\big) \\
 && 
 +\mu(\lambda_{0},\lambda_{1},|\lambda_{2}'|) \big(|\lambda_{2}|^{p-2}-|\lambda_{2}'|^{p-2}\big)
 \big(|\lambda_{2}|^2-|\lambda_{2}'|^2\big).
 \end{eqnarray*}
Finally, using   (\ref{m5}), (\ref{mlo}) and the monotonicity of the function $w\mapsto 
w^{p-2}$ on  $\mathbb{R}^{+}_*$ when  $p \geq 2$,  we infer that
\begin{eqnarray*} 
\bigl(\mathcal{F}(\lambda_{0},\lambda_{1},\lambda_{2})-\mathcal{F}(\lambda_{0},\lambda_{1},\lambda_{2}
') \bigr):(\lambda_{2}-\lambda_{2}')
\geq 0.
\end{eqnarray*}

\noindent {\bf Case 2:} $1<p<2$.

\smallskip

Let us denote now
\begin{eqnarray*}
\alpha=2\mu(\lambda_{0},\lambda_{1},|\lambda_{2}|),\quad 
\beta=2\mu(\lambda_{0},\lambda_{1},|\lambda_{2}'|).
\end{eqnarray*}
We have
\begin{eqnarray*}
\bigl(\mathcal{F}(\lambda_{0},\lambda_{1},\lambda_{2})-\mathcal{F}(\lambda_{0},\lambda_{1},\lambda_{2}
') \bigr):(\lambda_{2}-\lambda_{2}')
= \bigl(\alpha|\lambda_{2}|^{p-2}\lambda_{2} -\beta |\lambda_{2}'|^{p-2}\lambda_{2}' \bigr):(\lambda_{2}
 -\lambda_{2}')
\end{eqnarray*}
and using   
$
2\lambda_{2}:\lambda_{2}'=|\lambda_{2}|^2+|\lambda_{2}'|^2-|\lambda_{2}
-\lambda_{2}'|^{2}
$ we obtain
\begin{eqnarray*}
\begin{array}{ll}
\displaystyle \bigl(\mathcal{F}(\lambda_{0},\lambda_{1},\lambda_{2})-\mathcal{F}(\lambda_{0},\lambda_{1},\lambda_{2}
') \bigr):(\lambda_{2}-\lambda_{2}')
\\
\displaystyle =
\frac{1}{2}\Big( \bigl(\alpha|\lambda_{2}|^{p-2}
 +\beta |\lambda_{2}'|^{p-2} \bigr)|\lambda_{2}-\lambda_{2}'|^2+ \bigl(|\lambda_{2}|^2 
 -|\lambda_{2}'|^2 \bigr) \bigl(\alpha|\lambda_{2}|^{p-2}-\beta|\lambda_{2}'|^{p-2} \bigr)\Big).
\end{array}
\end{eqnarray*}
Without loss of generality we may assume that $|\lambda_2| \ge |\lambda'_2|>0$. 
With  the triangle inequality $|\lambda_{2}-\lambda_{2}'|\geq  |\lambda_{2}|-|\lambda_{2}'| \ge 0$, we get 
\begin{eqnarray*}
&& \bigl(\mathcal{F}(\lambda_{0},\lambda_{1},\lambda_{2})-\mathcal{F}(\lambda_{0},\lambda_{1},\lambda_{2}
') \bigr):(\lambda_{2}-\lambda_{2}') \\
&& \geq 
\frac{1}{2}\Big( \bigl(\alpha|\lambda_{2}|^{p-2} +\beta |\lambda_{2}'|^{p-2} \bigr)
\bigl(|\lambda_{2}| - |\lambda_{2}'| \bigr)^2+ \bigl(|\lambda_{2}|^2 
 -|\lambda_{2}'|^2 \bigr) \bigl(\alpha|\lambda_{2}|^{p-2}-\beta|\lambda_{2}'|^{p-2} \bigr)\Big) \\
 && = \frac{1}{2} \bigl(|\lambda_{2}| - |\lambda_{2}'| \bigr) \\
 && \qquad \times \Big( \bigl(\alpha|\lambda_{2}|^{p-2} +\beta |\lambda_{2}'|^{p-2}\bigr)
\bigl(|\lambda_{2}| - |\lambda_{2}'| \bigr)+ \bigl(|\lambda_{2}| 
 + |\lambda_{2}'| \bigr) \bigl(\alpha|\lambda_{2}|^{p-2}-\beta|\lambda_{2}'|^{p-2} \bigr)\Big).
 \end{eqnarray*}
Let  
$\displaystyle{t=\frac{|\lambda_{2}|}{|\lambda_{2}'|}\geq1}$.  Then
\begin{eqnarray*}
&&  (\mathcal{F}(\lambda_{0},\lambda_{1},\lambda_{2})-\mathcal{F}(\lambda_{0},\lambda_{1},\lambda_{2}
')):(\lambda_{2}-\lambda_{2}') \\
&& \geq  \frac{1}{2} |\lambda_{2}'|^{p} (t-1 ) \big( (\alpha t^{p-2} +\beta ) (t-1)
 +( t + 1 )(\alpha t^{p-2}-\beta )\big)\\
&& =  |\lambda_{2}'|^{p} (t-1 ) (\alpha t^{p-1} - \beta) 
 \end{eqnarray*}
 and from  (\ref{m5}) we obtain
\[0<\beta=2\mu(\lambda_{0},\lambda_{1},|\lambda_{2}'|)\leq2\mu(\lambda_{0},\lambda_{1},
|\lambda_{2}|)=\alpha.\]
It follows that $\alpha t^{p-1} \ge \alpha \ge \beta$ and finally
\begin{eqnarray*}
\bigl(\mathcal{F}(\lambda_{0},\lambda_{1},\lambda_{2})-\mathcal{F}(\lambda_{0},\lambda_{1},\lambda_{2}') \bigr):(\lambda_{2}-\lambda_{2}') \geq 0.
\end{eqnarray*}
\end{proof}

By using the monotonicity of ${\mathcal F}$  we obtain 

\begin{corollary}\label{lemma2}
Let  $\theta\in L^{q}(\Omega)$ with $q \ge 1$ and $\upsilon \in {\bf L}^p(\Omega)$ with $p>1$.
Assume  that  (\ref{eqG}), (\ref{m5})-(\ref{mlo}) hold. 
Then the operator  $\mathbf{A}_{\theta, \upsilon}$
is bounded  and  pseudo-monotone.   
\end{corollary}

\begin{proof}
The proof is a straighforward adaptation of standard results (see for instance Chapter 8, Theorem 8.9 in  \cite{LeDret2013}) and is left to the reader.
\end{proof}

Next we prove an existence and uniqueness result for (\ref{flow_aux}).

 \begin{theorem} \label{th1} %mls} 
 Let  $\theta\in L^{q}(\Omega)$ with $q \ge 1$ and $\upsilon \in {\bf L}^p(\Omega)$ with $p>1$. 
Assume  that  (\ref{eqG})-(\ref{eqfk}), (\ref{m5})-(\ref{mlo}) hold. 
 Then problem (\ref{flow_aux}) admits an unique solution.
 \end{theorem}
 
  \begin{proof}
The functional  $\Psi$   is convex, proper and lower semi-continuous on $V_{0.div}^p$ and with Corollary \ref{lemma2} we already know that the operator  $ \mathbf{A}_{\theta, \upsilon}$   is bounded  and 
pseudo-monotone. Moreover $\mathbf{A}_{\theta, \upsilon}$ is coercive i.e there exists 
$\tilde \upsilon^{*} \in V_{0.div}^p $   such that 
$\Psi(\tilde \upsilon^{*})< + \infty$ 
and 
\[ \lim_{\Vert \tilde \upsilon \Vert_{1.p} \rightarrow 
+\infty}\frac{\bigl\langle\mathbf{A}_{\theta, \upsilon}(\tilde \upsilon), \tilde \upsilon-\tilde \upsilon^{*}
\bigr\rangle + \Psi(\tilde \upsilon)}
{\Vert \tilde \upsilon\Vert_{1.p}}= +\infty.\]
Indeed let us choose  $\tilde \upsilon^{*}=0$. We have 
\begin{eqnarray*}
\begin{array}{ll}
\displaystyle \bigl\langle\mathbf{A}_{\theta, \upsilon} \tilde \upsilon, \tilde \upsilon\rangle 
= \int_{\Omega} \mathcal{F} \bigl( \theta,\upsilon+G, D(\tilde \upsilon+G) \bigr):D(\tilde \upsilon+G)\,dx
\\
\displaystyle 
-\int_{\Omega} \mathcal{F} \bigl( \theta, \upsilon+G, D(\tilde \upsilon+G) \bigr):D(G)\,dx
\end{array}
\end{eqnarray*}
and with   (\ref{mlo})   
\begin{eqnarray} \label{estim1}
\begin{array}{ll}
\displaystyle \langle\mathbf{A}_{\theta, \upsilon} \tilde \upsilon , \tilde \upsilon \rangle 
 \geq 2\mu_{0} \int_{\Omega} \bigl|D(\tilde \upsilon +G) \bigr|^{p}\,dx 
 -2\mu_{1} \int_{\Omega} \bigl|D(\tilde \upsilon +G) \bigr|^{p-1}|D(G)|\,dx \\
  \displaystyle \qquad \qquad \geq 2\mu_{0} \Vert D( \tilde \upsilon+G) \Vert_{(L^p(\Omega))^{3 \times 3}}^{p}-2\mu_{1}\Vert 
\tilde \upsilon+G\Vert_{1.p}^{p-1}\Vert G\Vert_{1.p}\\
\displaystyle \geq  2\mu_{0} \big|\Vert D(\tilde \upsilon) \Vert_{(L^p(\Omega))^{3 \times 3}}-\Vert 
D(G)\Vert_{(L^p(\Omega))^{3 \times 3}}\big|^p 
- 2\mu_{1} \bigl(\Vert \tilde \upsilon\Vert_{1.p}+\Vert 
G\Vert_{1.p} \bigr)^{p-1}\Vert G\Vert_{1.p}.
\end{array}
\end{eqnarray}
So, whenever $\Vert  \tilde \upsilon \Vert _{1.p} \neq 0$
\begin{eqnarray*}\label{eq321} %hara'}
\begin{array}{ll}
%&
\displaystyle \frac{\langle\mathbf{A}_{\theta, \upsilon} \tilde \upsilon, \tilde \upsilon\rangle}{\Vert \tilde \upsilon\Vert_{1.p}}
 \geq
 \Vert \tilde \upsilon\Vert_{1.p}^{p-1}
 \left( 2\mu_{0} \left| \frac{\Vert D(\tilde \upsilon) \Vert_{(L^p(\Omega))^{3 \times 3}} }{\Vert \tilde \upsilon\Vert_{1.p}} - \frac{\Vert D(G)\Vert_{(L^p(\Omega))^{3 \times 3}}}{\Vert \tilde \upsilon\Vert_{1.p}} \right|^{p} 
\right. \\
\displaystyle
\qquad \qquad \qquad \qquad \qquad 
\left. - \frac{2\mu_{1}}{\Vert \tilde \upsilon\Vert_{1.p}}
 \left( 1+ \frac{\Vert G\Vert_{1.p}}{\Vert \tilde \upsilon\Vert_{1.p}}\right)^{p-1}
 \Vert G\Vert_{1.p} \right).
\end{array}
\end{eqnarray*}

By Korn's  inequality (\cite{mosolov}), there exists  $C_{Korn}>0$  such that 
 \begin{eqnarray}\label{eq320}
% | u|_{1.p} : =
\Vert D(u) \Vert_{(L^p(\Omega))^{3 \times 3}} = \left( \int_{\Omega} \bigl|D(u) \bigr|^p\,dx \right)^{\frac{1}{p}} \geq C_{Korn}\Vert u \Vert_{1.p}\quad \forall u\in V^p_0.
%{\bf W}^{1,p}_{\Gamma_{1} \cup \Gamma_L} (\Omega).
 \end{eqnarray}
%where  $|u|_{1.p}^p=\int_{\Omega}|D(u)|^p\,dx$. 
Recalling that $\Psi(\tilde \upsilon) \ge 0$ for all $\tilde  \upsilon \in V_{0. div}^{p}$ and $p>1$ we get
\[
\lim_{\Vert \tilde \upsilon\Vert_{1.p} \rightarrow 
+\infty}\frac{\langle\mathbf{A}_{\theta, \upsilon} \tilde  \upsilon, \tilde  \upsilon\rangle + \Psi(\tilde  \upsilon)}
{\Vert \tilde  \upsilon\Vert_{1.p}}= +\infty.\]
Then by applying monotonicity arguments (see Chapter 2, Theorem 8.5 in \cite{Lions69}) we infer that problem (\ref{flow_aux}) admits at least a  solution.

\smallskip

Let us prove its uniqueness by a contradiction argument. So let $\tilde \upsilon_1 \in  V_{0.div}^{p}$ and $\tilde \upsilon_2 \in  V_{0.div}^{p}$ be two solutions of (\ref{flow_aux}). By choosing $\varphi = \tilde \upsilon_2$ then $\varphi =\tilde \upsilon_1$ we obtain
\begin{eqnarray} \label{ineg1}
\bigl\langle \mathbf{A}_{\theta, \upsilon} (\tilde \upsilon_1) - \mathbf{A}_{\theta, \upsilon} (\tilde \upsilon_2), \tilde \upsilon_1 - \tilde \upsilon_2 \bigr\rangle \le 0.
\end{eqnarray}
But 
\begin{eqnarray} \label{ineg2}
\begin{array}{ll}
\displaystyle \bigl\langle \mathbf{A}_{\theta, \upsilon} (\tilde \upsilon_1) - \mathbf{A}_{\theta, \upsilon} (\tilde \upsilon_2), \tilde \upsilon_1 - \tilde \upsilon_2 \bigr\rangle \\
\displaystyle
= \mu_0 \int_{\Omega} \bigl( \bigl|D( \tilde \upsilon_1 +G) \bigr|^{p-2} D(\tilde \upsilon_1+ G) 
- \bigl|D( \tilde \upsilon_2 +G) \bigr|^{p-2} D(\tilde \upsilon_2+ G) \Bigr) : D( \tilde \upsilon_1 - \tilde \upsilon_2) \, dx \\
\displaystyle 
+ \int_{\Omega} \Bigl( {\overline {\mathcal F}} \bigl( \theta, \upsilon +G, D(\tilde \upsilon_1  +G) \bigr)
- {\overline {\mathcal F}} \bigl( \theta, \upsilon + G , D(\tilde \upsilon_2  +G) \bigr) \Bigr)  : D( \tilde \upsilon_1 - \tilde \upsilon_2) \, dx
\end{array}
\end{eqnarray}
where 
\begin{eqnarray*}
{\overline{\mathcal{F}} } (\lambda_{0}, \lambda_{1},\lambda_{2})
= 2 {\overline \mu} \bigl(\lambda_{0} ,\lambda_{1},|\lambda_{2}| \bigr)|\lambda_{2}|^{p-2}\lambda_{2}
 \quad \hbox{\rm if }  \lambda_2 \not = 0_{\mathbb{R}^{3\times3}}, 
\quad {\overline{\mathcal{F}} } (\lambda_{0}, \lambda_{1},\lambda_{2})= 0_{\mathbb{R}^{3\times3}}
 \quad \hbox{\rm otherwise}
\end{eqnarray*}
and $\displaystyle {\overline \mu} = \mu - \frac{\mu_0}{2}$. Since ${\overline \mu}$ satisfies 
\begin{eqnarray*} 
\begin{array}{ll}
%\displaystyle
%  (o,e,d) \mapsto {\overline \mu} (o,e,d)\quad \mbox{is continuous on} 
% \  \mathbb{R}\times\mathbb{R}^{3}\times \mathbb{R}_{+}, \\
\displaystyle   d\mapsto {\overline \mu} (\cdot, \cdot,d) \quad \mbox{is monotone  increasing   on} 
\  \mathbb{R_{+}}, \\
\displaystyle  0<\frac{\mu_{0}}{2} \leq {\overline \mu} (o,e,d) \leq \mu_{1} - \frac{\mu_0}{2} 
\ \hbox{\rm for all} \  (o,e,d)\in \mathbb{R}\times\mathbb{R}^{3}\times\mathbb{R}_{+},
 \end{array}
\end{eqnarray*}
we infer with the same arguments as in Lemma \ref{lemma1} that $\lambda_2 \mapsto {\overline{\mathcal F}} (\cdot, \cdot, \lambda_2)$ is  monotone on $\mathbb{R}^{3\times3}$. Similarly we define 
\begin{eqnarray*}
{\mathcal{F}}_0 (\lambda_{0}, \lambda_{1},\lambda_{2})
=  \mu_0 |\lambda_{2}|^{p-2}\lambda_{2}
 \quad \hbox{\rm if }  \lambda_2 \not = 0_{\mathbb{R}^{3\times3}}, 
\quad {\mathcal{F}}_0 (\lambda_{0}, \lambda_{1},\lambda_{2})= 0_{\mathbb{R}^{3\times3}}
 \quad \hbox{\rm otherwise}
\end{eqnarray*}
and we obtain that $\lambda_2 \mapsto {\mathcal F}_0 (\cdot, \cdot, \lambda_2)$ is  monotone on $\mathbb{R}^{3\times3}$.
Hence 
\begin{eqnarray*}
\int_{\Omega} \Bigl( {\overline {\mathcal F}} \bigl( \theta, \upsilon +G, D(\tilde \upsilon_1  +G) \bigr)
- {\overline {\mathcal F}} \bigl( \theta, \upsilon + G, D(\tilde \upsilon_2  +G) \bigr) \Bigr)  : D( \tilde \upsilon_1 - \tilde \upsilon_2) \, dx \ge 0
\end{eqnarray*}
and 
\begin{eqnarray*}
\mu_0 \int_{\Omega} \bigl( \bigl|D( \tilde \upsilon_1 +G) \bigr|^{p-2} D(\tilde \upsilon_1+ G) 
- \bigl|D( \tilde \upsilon_2 +G) \bigr|^{p-2} D(\tilde \upsilon_2+ G) \Bigr) : D( \tilde \upsilon_1 - \tilde \upsilon_2) \, dx \ge 0.
\end{eqnarray*}
With (\ref{ineg1})-(\ref{ineg2}) we infer that
\begin{eqnarray} \label{ineg3}
 \int_{\Omega} \bigl( \bigl|D( \tilde \upsilon_1 +G) \bigr|^{p-2} D(\tilde \upsilon_1+ G) 
- \bigl|D( \tilde \upsilon_2 +G) \bigr|^{p-2} D(\tilde \upsilon_2+ G) \Bigr) : D( \tilde \upsilon_1 - \tilde \upsilon_2) \, dx =0
\end{eqnarray}
Let us distinguish now two cases.

\smallskip

{\bf Case 1:} $p \ge 2$.

\smallskip

For all $(\lambda, \lambda') \in {\mathbb R}^{3 \times 3} \times {\mathbb R}^{3 \times 3}$ we have
\begin{eqnarray} \label{ineg_case1}
\bigl( |\lambda|^{p-2} \lambda - | \lambda' |^{p-2} \lambda' ) : ( \lambda - \lambda' ) \ge \frac{1}{2^{p-1}}  |\lambda - \lambda' |^p .
\end{eqnarray}
Indeed if $\lambda = \lambda' =0_{{\mathbb R}^{3 \times 3}}$ the result is obvious. Otherwise with the same kind of computations as in Lemma \ref{lemma1} we get 
\begin{eqnarray*}
\begin{array}{ll}
\displaystyle \bigl( |\lambda|^{p-2} \lambda - | \lambda' |^{p-2} \lambda' ) : ( \lambda - \lambda' )  
\\
\displaystyle 
= \frac{1}{2} \Bigl( \bigl(   |\lambda|^{p-2} + |\lambda '|^{p-2} \bigr) |\lambda - \lambda' |^2 
 +  \bigl(   |\lambda|^2 - | \lambda' |^2 \bigr)  \bigl(   |\lambda|^{p-2} - |\lambda '|^{p-2} \bigr) \Bigr) \\
 \displaystyle 
 \ge \frac{1}{2}  \bigl(   |\lambda|^{p-2} + |\lambda '|^{p-2} \bigr) |\lambda - \lambda' |^2 .
 \end{array}
 \end{eqnarray*}
 Since $|\lambda| + |\lambda'| \not=0$ we obtain
 \begin{eqnarray*}
\begin{array}{ll}
\displaystyle \bigl( |\lambda|^{p-2} \lambda - | \lambda' |^{p-2} \lambda' ) : ( \lambda - \lambda' )  
\ge \frac{1}{2} \frac{ |\lambda|^{p-2} + |\lambda '|^{p-2} }{ \bigl( |\lambda| + |\lambda'| \bigr)^{p-2}} 
\bigl( |\lambda| + |\lambda'| \bigr)^{p-2} |\lambda - \lambda' |^2 \\
\displaystyle \ge  \frac{1}{2} \frac{ |\lambda|^{p-2} + |\lambda '|^{p-2} }{ \bigl( |\lambda| + |\lambda'| \bigr)^{p-2}} |\lambda - \lambda' |^p.
\end{array}
\end{eqnarray*}
But
\begin{eqnarray*}
\left( \frac{ |\lambda| + |\lambda'| }{2} \right)^{p-2} \le \Bigl( \max \bigl( |\lambda| , |\lambda'| \bigr) \Bigr)^{p-2} \le |\lambda|^{p-2} + |\lambda'|^{p-2}
\end{eqnarray*}
and thus
\begin{eqnarray*}
\displaystyle \bigl( |\lambda|^{p-2} \lambda - | \lambda' |^{p-2} \lambda' ) : ( \lambda - \lambda' ) 
\ge \frac{1}{2^{p-1}}  |\lambda - \lambda' |^p.
\end{eqnarray*}
By replacing $\lambda = D(\tilde \upsilon_1 +G)$, $\lambda' = D( \tilde \upsilon_2 +G)$ we infer from (\ref{ineg3}) that 
\begin{eqnarray*}
0 & = & \int_{\Omega} \bigl( \bigl|D( \tilde \upsilon_1 +G) \bigr|^{p-2} D(\tilde \upsilon_1+ G) 
- \bigl|D( \tilde \upsilon_2 +G) \bigr|^{p-2} D(\tilde \upsilon_2+ G) \Bigr) : D( \tilde \upsilon_1 - \tilde \upsilon_2) \, dx \\
& \ge & \frac{1}{2^{p-1}}  \Vert  D( \tilde \upsilon_1 - \tilde \upsilon_2) \Vert^p_{(L^p( \Omega))^{3 \times3}}.
\end{eqnarray*}
and with Korn's inequality we may conclude that $\tilde \upsilon_1 = \tilde \upsilon_2$.

\smallskip

\noindent {\bf Case 2:} $1< p< 2$.

\smallskip

In this case we have 
\begin{eqnarray} \label{inegp}
\displaystyle \bigl( |\lambda| + |\lambda'| \bigr)^{2-p}  \bigl( |\lambda|^{p-2} \lambda - | \lambda' |^{p-2} \lambda' ) : ( \lambda - \lambda' ) 
\ge (p-1)  |\lambda - \lambda' |^2
\end{eqnarray}
for all $(\lambda, \lambda') \in {\mathbb R}^{3 \times 3} \times {\mathbb R}^{3 \times 3}$. Indeed if $|\lambda| = |\lambda'|$ we have
\begin{eqnarray*}
\displaystyle \bigl( |\lambda| + |\lambda'| \bigr)^{2-p}  \bigl( |\lambda|^{p-2} \lambda - | \lambda' |^{p-2} \lambda' ) : ( \lambda - \lambda' ) 
= 2^{2-p}  |\lambda - \lambda' |^2
\end{eqnarray*}
and the conclusion follows from the inequality $2^{2-p} > p-1$ for all $p \in (1,2)$. Otherwise, if $|\lambda| \not= |\lambda'|$ we let
\begin{eqnarray*}
G(\lambda, \lambda') = \frac{\bigl( |\lambda| + |\lambda'| \bigr)^{2-p}  \bigl( |\lambda|^{p-2} \lambda - | \lambda' |^{p-2} \lambda' ) : ( \lambda - \lambda' ) }{ |\lambda - \lambda' |^2}.
\end{eqnarray*}
With the same computations as in Lemma \ref{lemma1} we have
\begin{eqnarray*}
\begin{array}{ll}
\displaystyle G(\lambda, \lambda') = \frac{1}{2} \frac{ \bigl( |\lambda| + |\lambda'| \bigr)^{2-p}}{ |\lambda - \lambda' |^2}
\Bigl( \bigl(   |\lambda|^{p-2} + |\lambda '|^{p-2} \bigr) |\lambda - \lambda' |^2 
 +  \bigl(   |\lambda|^2 - | \lambda' |^2 \bigr)  \bigl(   |\lambda|^{p-2} - |\lambda '|^{p-2} \bigr) \Bigr)
\\
\displaystyle
\ge \frac{ \bigl( |\lambda| + |\lambda'| \bigr)^{2-p}}{2} 
\left( \bigl(   |\lambda|^{p-2} + |\lambda '|^{p-2} \bigr) 
 +  \frac{ \bigl(   |\lambda| + | \lambda' | \bigr)  \bigl(   |\lambda|^{p-2} - |\lambda '|^{p-2} \bigr)}{  \bigl|   |\lambda| - | \lambda' | \bigr|} \right).
 \end{array}
 \end{eqnarray*}
 Without loss of generality we may assume that $|\lambda| > |\lambda'|$ and we let $\displaystyle t = \frac{|\lambda| }{ |\lambda'|} >1$. Thus
 \begin{eqnarray*}
\begin{array}{ll}
\displaystyle G(\lambda, \lambda') \ge  \frac{ \bigl( 1+t \bigr)^{2-p}}{2} 
\left( (  1+t^{p-2}) 
 +  \frac{ (1+t)  (t^{p-2} -1 )}{  |t-1 |} \right) \\
 \displaystyle = \frac{  (1+t)^{2-p} (t^{p-1} -1)}{t-1} \ge \frac{t (1-t^{1-p})}{t-1} = 1 - \frac{t^{2-p} -1}{t-1}.
 \end{array}
 \end{eqnarray*}
But, for all $t >1$ we have $\displaystyle \frac{t^{2-p} -1}{t-1} < 2-p$ and (\ref{inegp}) is satisfied.

Hence 
\begin{eqnarray*}
\displaystyle \Bigl( \bigl( |\lambda| + |\lambda'| \bigr)^p \Bigr)^{\frac{2-p}{2}} \Bigl( \bigl( |\lambda|^{p-2} \lambda - | \lambda' |^{p-2} \lambda' ) : ( \lambda - \lambda' ) \Bigr)^{\frac{p}{2}}
\ge (p-1)^{\frac{p}{2}}  |\lambda - \lambda' |^p.
\end{eqnarray*}
Since $p>1$ we have also 
\begin{eqnarray*}
 \bigl( |\lambda| + |\lambda'| \bigr)^p \le 2^{p-1}  \bigl( |\lambda|^p + |\lambda'|^p \bigr)
 \end{eqnarray*}
 which yields
 \begin{eqnarray*}
\displaystyle  \bigl( |\lambda|^p + |\lambda'|^p \bigr)^{\frac{2-p}{2}} \Bigl( \bigl( |\lambda|^{p-2} \lambda - | \lambda' |^{p-2} \lambda' ) : ( \lambda - \lambda' ) \Bigr)^{\frac{p}{2}}
\ge C_p |\lambda - \lambda' |^p \ge 0
\end{eqnarray*}
for all $(\lambda, \lambda') \in {\mathbb R}^{3 \times 3} \times {\mathbb R}^{3 \times 3}$, with $\displaystyle C_p = \frac{(p-1)^{\frac{p}{2}}}{ 2^{ \frac{(p-1)(2-p)}{2} } }$. By replacing $\lambda = D(\tilde \upsilon_1 +G)$, $\lambda' = D( \tilde \upsilon_2 +G)$ and using H\"older's inequality we obtain:
\begin{eqnarray*}
\begin{array}{ll}
\displaystyle C_p \int_{\Omega} \bigl| D( \tilde \upsilon_1 - \tilde \upsilon_2) \bigr|^p \, dx \\
\displaystyle 
\le \left( \int_{\Omega} \bigl( \bigl|D( \tilde \upsilon_1 +G) \bigr|^{p-2} D(\tilde \upsilon_1+ G) 
- \bigl|D( \tilde \upsilon_2 +G) \bigr|^{p-2} D(\tilde \upsilon_2+ G) \Bigr) : D( \tilde \upsilon_1 - \tilde \upsilon_2) \, dx \right)^{\frac{p}{2}} \\
\displaystyle \qquad \times \left( \int_{\Omega} \Bigl(  \bigl| D( \tilde \upsilon_1 + G ) \bigr|^p +  \bigl| D( \tilde \upsilon_2 + G ) \bigr|^p \Bigr) \, dx \right)^{\frac{2-p}{2}}.
\end{array}
\end{eqnarray*}
With (\ref{ineg3}) we get
\begin{eqnarray*}
C_p \bigl\|  D( \tilde \upsilon_1 - \tilde \upsilon_2) \bigr\|^p_{ (L^p(\Omega))^{3 \times 3}}
 = C_p \int_{\Omega} \bigl| D( \tilde \upsilon_1 - \tilde \upsilon_2) \bigr|^p \, dx  \le 0
 \end{eqnarray*}
 and we may conclude once again that $\tilde \upsilon_1 = \tilde \upsilon_2$.

 \end{proof}

 Having in mind the study of the coupled problem (P) we establish now a priori estimates, independent of the given temperature and velocity $(\theta, \upsilon)$ for the solutions of the flow problem (\ref{flow_aux}).

 %%%%%%%%%%%%%%%%%%%%%%%%%%%%%%%%%%%%%%%%%%%
 
 \begin{proposition}\label{prop1}
 Under the assumptions of Theorem \ref{th1},
 there exists a positive real number  $\mathcal{C}_{flow}$, independent of $\theta$ and $\upsilon$ and depending only on $f$, $G$, $k$, $s$, $\mu_0$ and $\mu_1$, 
such that the   solution  $\tilde \upsilon$ of problem (\ref{flow_aux}) satisfies
\begin{eqnarray}\label{eq5-13}
\Vert  \tilde \upsilon \Vert _{1.p} \leq \mathcal{C}_{flow}.
\end{eqnarray}
 \end{proposition}

\begin{proof}
Let  $\tilde \upsilon$ be a solution  of problem (\ref{flow_aux}). With $\varphi =0$  
 we obtain
\begin{eqnarray*}\label{sosa}
\langle\mathbf{A}_{\theta, \upsilon} \tilde \upsilon, \tilde \upsilon \rangle \le 
 \langle\mathbf{A}_{\theta, \upsilon} \tilde \upsilon, \tilde \upsilon \rangle 
 + \Psi(\tilde \upsilon) \leq \int_{\Omega} f \cdot  \tilde \upsilon \,dx+\Psi(0).
\end{eqnarray*}
By using (\ref{estim1}) and Poincar\'e's inequality we get
\begin{eqnarray*}
\begin{array}{ll}
\displaystyle 2 \mu_{0} \Big|\Vert D( \tilde \upsilon)  \Vert _{(L^p(\Omega))^{3 \times 3} }-\Vert D( G ) \Vert _{(L^p(\Omega))^{3 \times 3}}
\Big|^{p} 
\\
\displaystyle \leq  C_{Poincare} \Vert f\Vert_{\textbf{L}^{p'}(\Omega)}\Vert \tilde  \upsilon \Vert_{1.p}+\Psi(0)
+2\mu_{1} \Big(\Vert  \tilde \upsilon \Vert_{1.p}+\Vert G\Vert_{1.p}\Big)^{p-1}\Vert G \Vert_{1.p}.
\end{array}
\end{eqnarray*}
where $C_{Poincare}$ denotes the Poincar\'e's constant in ${\bf W}^{1,p}_{\Gamma_1 \cup \Gamma_L}(\Omega)$.  
If $\Vert \tilde  \upsilon \Vert _{1.p} \neq 0$ 
we obtain
\begin{eqnarray}\label{eq5-11}
\begin{array}{ll}
\displaystyle   C_{Korn} 
 \leq  \left( \frac{1}{2 \mu_0} C_{Poincare} \Vert f\Vert_{\textbf{L}^{p'}(\Omega)}
\Vert \tilde \upsilon \Vert_{1.p}^{1-p} + \frac{\Psi(0)}{\Vert 
\tilde \upsilon \Vert _{1.p}^p}  
+
\frac{2\mu_{1}}{\Vert  \tilde \upsilon \Vert_{1.p}}
\left(1+\displaystyle{\frac{\Vert G\Vert _{1.p}}{\Vert \tilde  \upsilon \Vert _{1.p}}}\right)^{p-1}\Vert G \Vert_{1.p} \right)^{1/p} 
\\
\displaystyle   \qquad \qquad  +  \frac{\Vert D(G) \Vert _{(L^p(\Omega))^{3 \times 3}}}{\Vert 
\tilde \upsilon \Vert_{1.p}} .
\end{array}
\end{eqnarray}
By observing that the mapping 
\begin{eqnarray*}
\begin{array}{ll}
\displaystyle \Lambda: t \mapsto 
\left( \frac{1}{2 \mu_0} C_{Poincare} \Vert f\Vert_{\textbf{L}^{p'}(\Omega)}
t^{1-p} +
\frac{\Psi(0)}{t^p}  
+
\frac{2\mu_{1}}{t}
\left(1+\displaystyle{\frac{\Vert G\Vert _{1.p}}{t}}\right)^{p-1}\Vert G \Vert_{1.p} \right) 
\\
\displaystyle \qquad \qquad 
- 
\frac{\Vert D(G) \Vert _{(L^p(\Omega))^{3 \times 3} } }{t}
- C_{Korn}
\end{array} 
\end{eqnarray*}
tends to $-  C_{Korn}  <0$ as $t$ tends to $+ \infty$ we infer that there exists $\mathcal{C}_{flow} >0$ such that $\Lambda(t) <0$ for all $t > \mathcal{C}_{flow}$. With (\ref{eq5-11}) it follows that 
\begin{eqnarray*}\label{eq5-13}
\Vert \tilde  \upsilon \Vert _{1.p} \leq \mathcal{C}_{flow}.
\end{eqnarray*}
\end{proof}

\begin{remark} The reader may wonder why we have not considered the fluid flow problem for a given temperature i.e.  the following variational inequality 
\begin{eqnarray}\label{flow_problem}
\left\{
\begin{array}{ll}
\hbox{\rm Find $\upsilon_{\theta} \in V_{0.div}^{p}$ such that} \\
\displaystyle
a(\theta;\upsilon_{\theta} , \varphi-\upsilon_{\theta})
+\Psi(\varphi)-\Psi(\upsilon_{\theta})
\geq \int_{\Omega} f\cdot (\varphi-\upsilon_{\theta})\,dx,
\quad \forall \varphi \in V_{0.div}^{p}
\end{array}
 \right.
\end{eqnarray} 
for a given temperature $\theta \in L^q(\Omega)$ with $q \ge 1$ and $p>1$. Indeed with the same arguments as in Corollary \ref{lemma2} and Theorem \ref{th1} we obtain that the operator $\mathbf{A}_{\theta} :V_{0.div}^{p}\rightarrow \bigl(V_{0.div}^{p} \bigr)'$ given by
\begin{eqnarray*}
%\begin{array}
\displaystyle \langle\mathbf{A}_{\theta} \upsilon , {\varphi}\rangle & = & a(\theta; \upsilon ,\varphi) \\
& = & 
\displaystyle \int _{\Omega} {\mathcal F} \bigl( \theta, \upsilon  +G, D( \upsilon  +G) \bigr): D(\varphi)dx \quad \forall ( \upsilon , \varphi) \in  V_{0.div}^{p} \times V_{0.div}^{p}
%\end{array}
\end{eqnarray*}
is bounded, pseudo-monotone and coercive. Hence problem (\ref{flow_problem}) admits a solution but we can not prove its uniqueness  and the fixed point technique introduced in Section \ref{section5} fails.
\end{remark}

%%%%%%%%%%%%%%%%%%%%%%%%%%%%%%%%%%%%%%%%%%%%%%%%%%%%%%%%%%%%%%%%%%%%%%%%%%%
%                    Debut de la section 4
%%%%%%%%%%%%%%%%%%%%%%%%%%%%%%%%%%%%%%%%%%%%%%%%%%%%%%%%%%%%%%%%%%%%%%%%%%

% \renewcommand{\theequation}{4.\arabic{equation}}
% \setcounter{equation}{0}
\section{Heat transfer auxiliary  problem} \label{section4}

As already outlined in the Introduction and in Section \ref{section2}, the right hand side of the heat equation (\ref{TEM1}) contains a term belonging to $L^1(\Omega)$, namely $2\mu(\theta,u ,|D(u)|)|D(u)|^p$. It is well known that for elliptic equations of the form 
\begin{eqnarray*}
\left\{
\begin{array}{ll}
\displaystyle - {\rm div} \bigl( a(x,\nabla \theta) \bigr)=g \quad \hbox{\rm in $\Omega$} \\
\displaystyle \theta = 0 \quad \hbox{\rm in $\partial \Omega$}
\end{array}
\right.
\end{eqnarray*}
with $g \in L^1(\Omega)$ and coercivity properties $a(x, \xi) \cdot \xi \ge \alpha_a |\xi|^2$ ($\alpha_a>0$) for all $\xi \in \mathbb{R}^3$ and for almost every $x \in \Omega$, we may expect solutions $\theta \in W^{1,q}_0(\Omega)$ with $1 \le q < 3/2$ and uniform estimates of  $\Vert  \theta \Vert _{W^{1,q} (\Omega)}$ whenever $g$ remains in a given ball of $L^1(\Omega)$ (see \cite{Bocca}  and the references therein). Since we consider also the convection term in (\ref{TEM1}) and mixed Dirichlet-Neumann boundary conditions (\ref{TEM3}) we can not apply directly these results. In order to cope with the difficulty due to the $L^1$-term in the right-hand side of (\ref{TEM1}) we 
follow the same strategy as in \cite{Bocca2} and we 
replace it by some approximate function $g_{\delta} (\theta, \upsilon)$ given by 
\begin{eqnarray*}
 g_{\delta}(\theta,\upsilon)
=\frac{2\mu \bigl(\theta,\upsilon +G, \bigl|D(\upsilon+G) \bigr| \bigr) \bigl|D(\upsilon+G) \bigr|^p }{ 1 + 2\delta\mu \bigl(\theta, \upsilon +G, \bigl|D( \upsilon+G) \bigr| \bigr) \bigl|D( \upsilon+G) \bigr|^p}
\end{eqnarray*}
with $\delta >0$. Obviously for any $\delta >0$, $\upsilon \in V_{0. div}^{p}$ with $p>1$ and $\theta \in L^q(\Omega)$ with $q \ge 1$ we have $g_{\delta} (\theta, \upsilon) \in L^{\infty} (\Omega) $ and 
\begin{eqnarray} \label{m_delta_1}
\Vert   g_{\delta}(\theta, \upsilon) \Vert _{L^{\infty}(\Omega)} \le \frac{1}{\delta}.
\end{eqnarray}
Moreover we may observe that (\ref{mlo}) implies that $ g_{\delta}(\theta,\upsilon)$ is also  bounded in $L^1(\Omega)$ 
uniformly with respect to $\theta$ and $\delta$ and we have 
\begin{eqnarray} \label{m_delta_2}
\Vert g_{\delta}(\theta, \upsilon) \Vert_{L^1(\Omega)} 
 \leq 2 \mu_{1}\Vert D( \upsilon +G) \Vert_{(L^p(\Omega))^{3\times3}}^p 
\end{eqnarray}
for all $\delta >0$, for all $\theta \in L^q(\Omega)$ and for all $\upsilon \in V_{0 .div}^{p}$ with $p>1$ and  $q \ge 1$. Similarly we define $\theta^b_{\delta}$ as
\begin{eqnarray}  \label{theta_b_delta}
\theta^b_{\delta} = \frac{\theta^b}{1 + \delta | \theta^b|}
\end{eqnarray}
for all $\delta >0$.

\smallskip

So we consider the following approximate heat equation:
\begin{eqnarray} \label{heat_approx1}
 (\upsilon +G) \cdot \nabla\theta_{\delta} - {\rm  div} (K \nabla \theta_{\delta}) 
 = g_{\delta}(\theta_{\delta}, \upsilon) 
+r (\theta_{\delta}) \quad \mbox{in} \  \Omega
\end{eqnarray}
with the boundary conditions 
\begin{eqnarray} \label{heat_approx1_bc}
\theta_{\delta}=0\quad \mbox{on}\  \Gamma_{1}\cup \Gamma_{L},  \quad 
( K \nabla \theta_{\delta}) \cdot n
=\theta^{b}_{\delta}
\quad \mbox{on}\  \Gamma_{0}.
\end{eqnarray}
Moreover let us recall that, since we expect $\theta \in W^{1,q}_{\Gamma_1 \cup \Gamma_L} (\Omega)$ with $1 \le q <3/2$,  we have also some condition on $p$ in order to ensure that the convection term $(\upsilon + G) \cdot \nabla \theta$ belongs to $L^1(\Omega)$ i.e. $\upsilon +G$ should belong to $L^{p_*}(\Omega)$ with $p_* \ge q'$ which yields the compatibility conditions between $p$ and $q$ given by (\ref{compa1}) or equivalently (\ref{compa2}).
%Since $\upsilon + G \in W^{1,p}(\Omega)$ this condition is fullfilled if $p \ge 3$ or if $\displaystyle \frac{3p}{3-p} \ge q^* $ if $1<p<3$. Thus if $q=1$ we should have $p >3$, and if $q \in (1, 3/2)$ we should have $\displaystyle p \ge \frac{3q}{4q-3}$.

\smallskip

Since problem (\ref{heat_approx1})-(\ref{heat_approx1_bc}) is non-linear we consider first a linearized problem and we prove an existence and uniqueness result.

\begin{proposition} \label{prop2}
Let $p \ge 3/2$, $\upsilon \in V_{0.div}^p$ and $G$ satisfy (\ref{eqG}). Let us assume that   (\ref{rop}) and  (\ref{TEM2})-(\ref{TEM2bis})-(\ref{Cr})-(\ref{thetab}) hold. Then, for all $\delta >0$ and $\tilde \theta \in L^2(\Omega)$, there exists an unique $\tilde \theta_{\delta} \in   W^{1,2}_{ \Gamma_{1}\cup \Gamma_{L}}(\Omega) $  such that
\begin{eqnarray}\label{Lpb} 
B(\upsilon, \tilde \theta_{\delta},w)=L_{\delta} (\upsilon, \tilde \theta ,w) \quad \forall w\in W^{1,2}_{ \Gamma_{1}\cup \Gamma_{L}}(\Omega)
\end{eqnarray}
where
\[B(\upsilon, \tilde \theta_{\delta},w)=  \int_{\Omega} \bigl( (\upsilon +G) \cdot \nabla \tilde \theta_{\delta} \bigr) w\,  dx
+ \int_{\Omega} \bigl( K \nabla \tilde \theta_{\delta} \bigr) \cdot \nabla w \, dx  , \]
\[L_{\delta} (\upsilon, \tilde \theta  , w)= \int _{\Omega} g_{\delta}( \tilde \theta  ,
{\upsilon}) w \,dx + \int _{\Omega} r( \tilde \theta  ) w \, dx + \int _{ \Gamma_{0}} \theta^{b}_{\delta}
w \, dx'.\]
\end{proposition}

\begin{remark} The first integral term in $B(\upsilon, \tilde \theta_{\delta},w)$ is well defined: indeed if $\tilde \theta_{\delta}$ and $w$ belong to $ W^{1,2}_{ \Gamma_{1}\cup \Gamma_{L}}(\Omega) $, then $\nabla \tilde \theta_{\delta} \in L^2(\Omega)$ and $w \in L^6(\Omega)$ hence $ w \nabla \tilde \theta_{\delta}  \in {\bf L}^{3/2}(\Omega)$ and $\upsilon + G \in L^3 (\Omega)$ for all $\upsilon \in V_{0.div}^p$ with $p \ge 3/2$.
\end{remark}

\begin{proof}
By using the continuous injection of $W^{1,2}_{ \Gamma_{1}\cup \Gamma_{L}}(\Omega) $ into $L^6 (\Omega)$ we obtain immediately that $B(\upsilon, \cdot, \cdot)$ is bilinear and continuous on $W^{1,2}_{ \Gamma_{1}\cup \Gamma_{L}}(\Omega) $. Moreover, for all $ w  \in W^{1,2}_{ \Gamma_{1}\cup \Gamma_{L}}(\Omega)$ we have 
\begin{eqnarray*}
 B(\upsilon, w, w)
 =  \int_{\Omega} \bigl( ( \upsilon+G)\cdot \nabla w \bigr) w \,dx
 + \int_{\Omega} ( K \nabla w) \cdot \nabla w \,dx .
%\geq  k_{0} \Vert\theta_{\delta}\Vert_{1.2}^2 
%  +\frac{1}{2}\int_{\Omega}\  (\upsilon+G)_{i} \partial_{i}(\theta_{\delta}^2 )\,dx.
\end{eqnarray*}
%and since $w^2 \in W^{1,3}_{\Gamma_1 \cup \Gamma_L} (\Omega)$ Green's formula yields
For any $w \in {\mathcal D} ({\overline \Omega})$ Green's formula yields
\begin{eqnarray*}
\int_{\Omega} \bigl( ( \upsilon+G)\cdot \nabla w \bigr) w \,dx
%= \frac{1}{2} \int_{\Omega}  (\upsilon+G) \cdot \nabla \bigl( w^2 \bigr)\,dx
& =  & - \frac{1}{2} \int_{\Omega} {\rm div}( \upsilon+G)  w^2 \, dx  + \frac{1}{2} \int_{\partial \Omega} \bigl( ( \upsilon + G) \cdot n \bigr)  w^2 \, dY \\
& = & \frac{1}{2} \int_{\partial \Omega} \bigl( ( \upsilon + G) \cdot n \bigr)  w^2 \, dY.
\end{eqnarray*}
By using the density of ${\mathcal D} ({\overline \Omega})$ into $W^{1,2}( \Omega)$ and the continuity of the trace operator from $W^{1,2}( \Omega)$ into $L^4(\partial \Omega)$ we obtain that the same equality holds for any $w \in W^{1,2}(\Omega)$.

Hence with (\ref{TEM2bis})
\begin{eqnarray*}
 B(\upsilon, w, w )   = \int_{\Omega} ( K \nabla w) \cdot \nabla w \,dx \geq  k_{0} \Vert w \Vert_{1.2}^2 
 \quad \forall w \in W^{1,2}_{\Gamma_1 \cup \Gamma_L} (\Omega).
 \end{eqnarray*}
Finally with (\ref{Cr})-(\ref{thetab}) and (\ref{m_delta_1})-(\ref{theta_b_delta}) we obtain that  $L_{\delta}  (\upsilon, \tilde \theta  , \cdot)$ is linear and continuous on $W^{1,2}_{ \Gamma_{1}\cup \Gamma_{L}}(\Omega) $ and the existence of a unique solution $\tilde \theta_{\delta}$ to (\ref{Lpb}) follows from Lax-Milgram theorem.
\end{proof}

Of course we have a trivial estimate of the solution of problem (\ref{Lpb}) in $ W^{1,2}_{ \Gamma_{1}\cup \Gamma_{L}}(\Omega) $. Indeed let $\tilde \theta \in L^2(\Omega)$ and $\upsilon \in V^p_{0.div}$. By choosing $w = \tilde \theta_{\delta}$ in (\ref{Lpb}) we obtain
\begin{eqnarray}  \label{estim2}
\begin{array}{ll}
\displaystyle k_{0} \Vert \tilde \theta_{\delta} \Vert_{1.2}^2 \le B(\upsilon, \tilde \theta_{\delta}, \tilde \theta_{\delta}) = L_{\delta} (\upsilon, \tilde \theta, \tilde \theta_{\delta}) 
\\
\displaystyle \le \left( C'_{Poincare} \left( \frac{1}{\delta} + \| r\|_{L^{\infty} (\mathbb{R})} \right) |\Omega|^{1/2} 
+ C_{\gamma} \frac{1}{\delta} |\Gamma_0|^{1/2}
%\| \theta^b_{\delta} \|_{L^2(\Gamma_0)} 
\right)
\Vert \tilde \theta_{\delta} \Vert_{1.2}
\end{array}
\end{eqnarray}
where $C_{\gamma}$ is the norm of the trace operator $\gamma : W^{1,2}_{ \Gamma_{1}\cup \Gamma_{L}} (\Omega) \to L^2 (\partial \Omega)$ and $C'_{Poincare}$ denotes the Poincar\'e's constant on $W^{1,2}_{ \Gamma_{1}\cup \Gamma_{L}}(\Omega)$. Hence
\begin{eqnarray*}
%\Vert \tilde \theta_{\delta} \Vert_{L^2(\Omega)} \le C'_{Poincare} 
\Vert \tilde \theta_{\delta} \Vert_{1.2} \le R_{\delta} := \frac{1}{k_0} \left( C'_{Poincare} \left( \frac{1}{\delta} + \| r\|_{L^{\infty} (\mathbb{R})} \right) |\Omega|^{1/2} + C_{\gamma}  \frac{1}{\delta} |\Gamma_0|^{1/2}
%\| \theta^b_{\delta} \|_{L^2(\Gamma_0)}
 \right).
\end{eqnarray*} 
Clearly $R_{\delta}$ is independent of $\tilde \theta$ and $\upsilon$ but depends on $\delta$. 
On the other hand, by using the uniform boundedness of $g_{\delta} (\tilde \theta, \upsilon)$ in $L^1(\Omega)$ with respect to $\tilde \theta$ and $\delta$ we may obtain an estimate, independent of $\tilde \theta$ and $\delta$ of $\tilde \theta_{\delta}$ in $W^{1,q}_{ \Gamma_{1}\cup \Gamma_{L}}(\Omega)$ for all $\displaystyle q \in \bigl[1, 3/2 \bigr)$. More precisely we have

\begin{proposition} \label{prop3}
Let $\displaystyle q \in \bigl[1, 3/2 \bigr)$ and $p \ge 3/2$.
%$p>3$ if $q=1$, $\displaystyle p \ge \frac{3q}{4q-3}$ otherwise.
Let $\tilde \theta \in L^2(\Omega)$,  $\upsilon \in V_{0.div}^p$,  $G$ satisfying (\ref{eqG}) and assume that  (\ref{rop})-(\ref{mlo}) and  (\ref{TEM2})-(\ref{TEM2bis})-(\ref{Cr})-(\ref{thetab}) hold. Then there exists a positive real number ${\mathcal C}_{heat}$, independent of $\tilde \theta$ and $\delta$ and depending only on $\mu_1 \Vert D(\upsilon+G) \Vert^p_{(L^p(\Omega))^{3 \times 3}}$, $r$, $\theta^b$, $k_0$ and $q$, such that the unique solution   $ \tilde \theta_{\delta}  $  of (\ref{Lpb}) satisfies
\begin{eqnarray} \label{estim3} 
\Vert \tilde  \theta_{\delta} \Vert_{1.q} \le {\mathcal C}_{heat}.
\end{eqnarray} 
\end{proposition}

\begin{proof}
Let $\phi: {\mathbb R} \to {\mathbb R}$ given by 
$\displaystyle \phi(t) = sg(t) \left( 1 - \frac{1}{ \bigl( 1 + |t| \bigr)^{\zeta} } \right)$ with $\zeta >0$ to be chosen later and $\displaystyle sg(t) = \frac{t}{|t|}$ if $t \not=0$ , $sg(0) = 0$. Then $\phi(0) =0$, $\phi \in C^1(\mathbb R; \mathbb R)$ and 
\begin{eqnarray} \label{phi1}
\bigl| \phi(t) \bigr| \le 1 , \quad \phi'(t) = \frac{\zeta}{\bigl( 1 + |t| \bigr)^{\zeta + 1} } \quad \forall t \in \mathbb R
\end{eqnarray}
and $\phi'$ is $\zeta(\zeta+1)$-Lipschitz continuous on ${\mathbb R}$.
It follows that for any $w \in W^{1,2}(\Omega)$ we have $\phi (w) \in W^{1,2} (\Omega)$ and 
\begin{eqnarray} \label{phi2}
\nabla \bigl( \phi(w) \bigr) = \phi'(w) \nabla w = \frac{\zeta}{\bigl( 1 + |w| \bigr)^{\zeta + 1} } \nabla w.
\end{eqnarray}
Moreover if $w \in W^{1,2}_{ \Gamma_{1}\cup \Gamma_{L}}(\Omega)$ then $\phi(w) \in W^{1,2}_{ \Gamma_{1}\cup \Gamma_{L}}(\Omega)$. Indeed let $(w_n)_{n \ge 1}$ be a sequence of ${\mathcal D} ({\overline \Omega})$ which converges strongly to $w$ in $W^{1,2}(\Omega)$. Then, with (\ref{phi1})-(\ref{phi2}), we obtain that $\bigl( \phi(w_n) \bigr)_{n \ge 1}$ converges strongly to $\phi(w)$ in $W^{1, 3/2} (\Omega)$. By continuity of the trace operator $\gamma$, the sequences $\bigl( \gamma (w_n) \bigr)_{n \ge 1}$ and $\bigl( \gamma \bigl( \phi (w_n) \bigr) \bigr)_{n \ge 1}$ converge also strongly to $\gamma(w)$ and $\gamma \bigl( \phi(w) \bigr)$ respectively in $L^2 (\partial \Omega)$.  Hence, possibly extracting a subsequence still denoted $(w_n)_{n \ge 1}$, we have 
\begin{eqnarray*}
\begin{array}{ll}
\displaystyle \gamma (w_n) = {w_n}_{| \partial \Omega} \to \gamma (w) \quad \hbox{\rm a.e. on} \ \partial \Omega, \\
\displaystyle \gamma \bigl( \phi (w_n) \bigr) = \phi(w_n)_{| \partial \Omega} \to \gamma \bigl( \phi (w) \bigr) \quad \hbox{\rm a.e. on} \ \partial \Omega.
\end{array}
\end{eqnarray*}
Owing that $w \in W^{1,2}_{ \Gamma_{1}\cup \Gamma_{L}}(\Omega)$ we get
\begin{eqnarray*}
{w_n}_{| \partial \Omega} \to \gamma (w)= 0  \quad \hbox{\rm a.e. on} \ \Gamma_1 \cup \Gamma_L.
\end{eqnarray*}
Thus
\begin{eqnarray*}
\displaystyle \phi(w_n)_{| \partial \Omega} \to 0 = \gamma \bigl( \phi (w) \bigr) \quad \hbox{\rm a.e. on} \  \Gamma_1 \cup \Gamma_L
\end{eqnarray*}
and with (\ref{phi1})
\begin{eqnarray*}
\bigl| \gamma \bigl( \phi (w) \bigr) \bigr| \le 1 \quad \hbox{\rm a.e. on} \  \Gamma_0.
\end{eqnarray*}

Let $\Phi: {\mathbb R} \to {\mathbb R}$ given by $\displaystyle \Phi(t) = \int_0^t \phi(s) \, ds$ for all $t \in {\mathbb R}$. For any $w \in W^{1,2}_{ \Gamma_{1}\cup \Gamma_{L}}(\Omega)$ we have also $\Phi (w) \in W^{1,2}_{ \Gamma_{1}\cup \Gamma_{L}}(\Omega)$ and
\begin{eqnarray} \label{phi3}
\nabla \bigl( \Phi(w) \bigr) = \Phi'(w) \nabla w = \phi(w) \nabla w.
\end{eqnarray}

Now let $\delta >0$, $\tilde \theta \in L^2(\Omega)$ and $\upsilon \in V^p_{0.div}$. By choosing $w = \phi (\tilde \theta_{\delta}) $ as a test-function in (\ref{Lpb}) we get
\begin{eqnarray}\label{F0}
\begin{array}{ll}
& \displaystyle{\int _{\Omega} \bigl( ( \upsilon+G) \cdot\nabla \tilde \theta_{\delta} \bigr) \phi(\tilde \theta_{\delta}) \, dx}
\displaystyle{ + \zeta\int _{\Omega}  \frac{ (K \nabla \tilde \theta_{\delta}) \cdot \nabla \tilde \theta_{\delta} }{(1+|\tilde \theta_{\delta}|)^{\zeta+1}} \, dx}   
 \medskip \\
&\displaystyle{ =\int _{\Omega} g_{\delta}( \tilde \theta , {\upsilon}) \phi(\tilde \theta_{\delta}) \, dx} 
  \displaystyle{ + \int _{\Omega} r(\tilde \theta ) \phi( \tilde \theta_{\delta}) \, dx}  
  \displaystyle{ + \int _{ \Gamma_{0}}\theta^{b}_{\delta} \gamma \bigl( \phi(\tilde \theta_{\delta}) \bigr) \,dx'.}
  \end{array}
\end{eqnarray}
With Green's formula we have 
\begin{eqnarray*}
%\begin{array}{ll} 
\displaystyle \int_{\Omega} \bigl( u \cdot \nabla \tilde \theta_{\delta} \bigr) \phi ( \tilde \theta_{\delta} ) \, dx  
= \int_{\Omega} u  \cdot \nabla \bigl( \Phi( \tilde \theta_{\delta}) \bigr) \, dx 
%\\
\displaystyle  =  - \int_{\Omega} {\rm div }(u) \Phi(\tilde \theta_{\delta} ) \, dx + \int_{\Gamma_0} u \cdot n \gamma  \bigl( \Phi ( \tilde \theta_{\delta}) \bigr) \, dx'
%\end{array}
\end{eqnarray*}
for all $u \in \bigl( {\mathcal D} ({\overline{\Omega}}) \bigr)^3 = {\bf { D}} ({\overline{\Omega}})$. Since $p \ge 3/2$ we have  ${\bf W}^{1,p}(\Omega) \subset {\bf L}^3(\Omega)$ with continuous injection and the trace operator maps continuously ${\bf W}^{1,p}(\Omega)$ into ${\bf L}^2 (\partial \Omega)$.
%and   $H^1 (\Omega) \subset L^6(\Omega)$, $\gamma \bigl( H^1 (\Omega) \bigr) \subset L^4(\partial \Omega)$. 
 By using the density of ${\bf { D}} ({\overline{\Omega}})$ into ${\bf W}^{1,p}(\Omega)$ we obtain that the same equality holds for all $u \in {\bf W}^{1,p}(\Omega)$. With $u = \upsilon +G$ we get finally 
\begin{eqnarray*}
\int _{\Omega} \bigl( ( \upsilon+G) \cdot\nabla \tilde \theta_{\delta} \bigr) \phi(\tilde \theta_{\delta}) \, dx =0.
\end{eqnarray*}
Going back to (\ref{F0}) and using (\ref{TEM2}) we get
\begin{eqnarray*} 
%\begin{array}{ll}
\displaystyle k_{0} \zeta \int _{\Omega}  \frac{|\nabla \tilde \theta_{\delta} |^2}{(1+|\tilde \theta_{\delta}|)^{\zeta+1}}\,dx
\le \int _{\Omega} \bigl| g_{\delta}( \tilde \theta , {\upsilon}) \bigr|  \, dx 
  +\int _{\Omega} \bigl| r( \tilde \theta ) \bigr|\, dx 
  +\int _{ \Gamma_{0}} \bigl| \theta^{b}_{\delta} \bigr|  \,dx' 
  %\\
 %\displaystyle 
 %\le 2 \mu_1 \Vert D(\upsilon + G) \Vert_{1.p}^p + |\Omega| \Vert r \Vert_{L^{\infty}(\mathbb R)} +  \Vert \theta^b \Vert_{L^1(\Gamma_0)} 
% \end{array}
 \end{eqnarray*} 
% where $C_{\gamma}$ denotes the norm of the trace operator $\gamma$ on $ W^{1,2}_{ \Gamma_{1}\cup \Gamma_{L}}(\Omega)$.
% But for all $\lambda >0$ we have
% \begin{eqnarray*}
% \begin{array}{ll}
%\displaystyle C_{\gamma} \Vert \theta_b \Vert_{L^2(\Gamma_0)} \Vert \phi (\theta_{\delta}) \Vert_{1.2} \le \frac{1}{2\lambda} \Vert \theta_b \Vert_{L^2(\Gamma_0)}^2 + \frac{\lambda}{2}  \Vert \phi (\theta_{\delta}) \Vert_{1.2}^2 
%%\\
%\displaystyle =  \frac{C_{\gamma}^2}{2\lambda} \Vert \theta_b \Vert_{L^2(\Gamma_0)}^2 
%+ \frac{\lambda}{2} \zeta^2 \int _{\Omega}  \frac{|\nabla\theta_{\delta} |^2}{(1+|\theta_{\delta}|)^{2(\zeta+1)}}\,dx
%\\
%\displaystyle
%\le \frac{C_{\gamma}^2}{2\lambda} \Vert \theta_b \Vert_{L^2(\Gamma_0)}^2  
%+ \frac{\lambda}{2} \zeta^2 \int _{\Omega}  \frac{|\nabla\theta_{\delta} |^2}{(1+|\theta_{\delta}|)^{\zeta+1}}\,dx.
%\end{array}
%\end{eqnarray*}
%Then we choose $\displaystyle \lambda = \frac{k_0}{\zeta}$ and we obtain
%\begin{eqnarray} \label{mir0}
%\displaystyle  \int _{\Omega}  \frac{|\nabla\theta_{\delta} |^2}{(1+|\theta_{\delta}|)^{\zeta+1}}\,dx
%  \le \frac{2}{k_0 \zeta} \left( 2 \mu_1 \Vert D(\upsilon + G) \Vert_{1.p}^p + |\Omega| \Vert r \Vert_{L^{\infty}(\mathbb R)} + \frac{\zeta C_{\gamma}^2}{2 k_0} \Vert \theta_b \Vert_{L^2(\Gamma_0)}^2 \right).
% \end{eqnarray} 
and thus
\begin{eqnarray} \label{mir0}
\displaystyle  \int _{\Omega}  \frac{|\nabla \tilde \theta_{\delta} |^2}{(1+|\tilde \theta_{\delta}|)^{\zeta+1}}\,dx
  \le \frac{1}{k_0 \zeta} \bigl( 2 \mu_1 \Vert D(\upsilon + G) \Vert_{(L^p(\Omega))^{3 \times 3}}^p + |\Omega| \Vert r \Vert_{L^{\infty}(\mathbb R)} +  \Vert \theta^b \Vert_{L^1(\Gamma_0)} \bigr).
 \end{eqnarray} 
Then we estimate $\Vert \tilde \theta_{\delta} \Vert_{1.q}^q$ as follows
 \begin{eqnarray*} 
 \Vert \tilde \theta_{\delta} \Vert_{1.q}^q =  \int _{\Omega} |\nabla \tilde \theta_{\delta} |^q\,dx 
  =   \int _{\Omega}  \frac{|\nabla \tilde \theta_{\delta}
  |^q}{(1+|\tilde \theta_{\delta}|)^{\frac{(\zeta+1)q}{2}}} 
  \bigl(  1+|\tilde \theta_{\delta}| \bigr)^{\frac{(\zeta+1)q}{2}} \,dx .
 \end{eqnarray*}
By using H\"older's inequality we obtain
 \begin{eqnarray*}\label{mir}
  \int _{\Omega} |\nabla \tilde \theta_{\delta} |^q \,dx
  \leq \left( \int _{\Omega} \frac{|\nabla \tilde \theta_{\delta} |^2}{(1+|\tilde \theta_{\delta}|)^{\zeta+1}} \, dx
  \right)^{\frac{q}{2}} \left( \int _{\Omega}  \bigl( 1+|\tilde \theta_{\delta}| \bigr)^{\frac{(\zeta+1)q}{2-q}} \, dx
  \right)^{\frac{2-q}{2}} .
 \end{eqnarray*}
 Now we observe that since $\displaystyle q \in \bigl[1, 3/2 \bigr)$ we have $W^{1,q}(\Omega) \subset L^{q_*}(\Omega)$ with $\displaystyle q_* = \frac{3q}{3-q}$ and we may choose $\zeta >0$ such that $\displaystyle \frac{(\zeta +1) q}{2-q} = q_*$, namely $\displaystyle \zeta = \frac{3 -2q}{3-q}$. With (\ref{mir0}) we get
 \begin{eqnarray*}
 \Vert \tilde \theta_{\delta} \Vert_{1.q} = \left( \int _{\Omega} |\nabla \tilde \theta_{\delta} |^q \,dx \right)^{\frac{1}{q}}
  \leq C_1 \left( \int _{\Omega} \bigl( 1+| \tilde \theta_{\delta}| \bigr)^{q_*} \, dx  \right)^{\frac{2-q}{2q}} 
 \end{eqnarray*} 
 with
 \begin{eqnarray} \label{mir1}
 C_1 =  \frac{1}{ \sqrt{k_0 \zeta}}  \Bigl( 2 \mu_1 \Vert D(\upsilon + G) \Vert_{(L^p(\Omega))^{3 \times 3} }^p + |\Omega| \Vert r \Vert_{L^{\infty}(\mathbb R)} +  \Vert \theta^b \Vert_{L^1(\Gamma_0)} \Bigr)^{1/2} .
 \end{eqnarray}
 Since $q_*>1$ we have 
 \begin{eqnarray*}
 \bigl(1+|\tilde \theta_{\delta}| \bigr)^{q_*} = 2^{q_*} \left( \frac{1}{2} + \frac{|\tilde \theta_{\delta}|}{2} \right)^{q_*} \le 2^{q_* -1}  \bigl(1+|\tilde \theta_{\delta}|^{q_*} \bigr)
 \end{eqnarray*}
 which yields
  \begin{eqnarray*}
 \Vert \tilde \theta_{\delta} \Vert_{1.q} \le C_1  2^{(q_*-1) \frac{2-q}{2q} }  \left( |\Omega| + \int_{\Omega} |\tilde \theta_{\delta}|^{q_*} \, dx  \right)^{\frac{2-q}{2q}}.
 \end{eqnarray*}
Recalling that $\displaystyle \frac{2-q}{2q} \in (0,1)$ we have
\begin{eqnarray*}
1 = \frac{a}{a+b} + \frac{b}{a+b} \le \left( \frac{a}{a+b} \right)^{\frac{2-q}{2q}} + \left( \frac{b}{a+b} \right)^{\frac{2-q}{2q}} \quad \forall (a,b) \in {\mathbb R}_*^+ \times {\mathbb R}^+.
\end{eqnarray*}
So
  \begin{eqnarray*}
 \Vert \tilde \theta_{\delta} \Vert_{1.q} \le C_1 2^{(q_*-1) \frac{2-q}{2q}} \left( |\Omega|^{\frac{2-q}{2q}} + \Vert \tilde \theta_{\delta} \Vert_{L^{q_*}(\Omega)}^{q_*\frac{2-q}{2q}}  \right).
 \end{eqnarray*}
 But $\displaystyle  \frac{2-q}{2q} < \frac{1}{q_*} = \frac{3-q}{3q}$ and $\displaystyle \alpha := q_* \frac{2-q}{2q}  \in (0,1)$. It follows that 
  \begin{eqnarray*}
 \Vert \tilde \theta_{\delta} \Vert_{1.q} = \Vert \tilde \theta_{\delta} \Vert_{1.q}^{\alpha} \Vert \tilde \theta_{\delta} \Vert_{1.q}^{1-\alpha} 
 \le C_1 2^{(q_*-1) \frac{2-q}{2q} } \left( |\Omega|^{\frac{2-q}{2q}} + C_q^{\alpha} \Vert \tilde \theta_{\delta} \Vert_{1.q}^{\alpha}  \right)
 \end{eqnarray*} 
 where $C_q$ denotes the norm of the canonical injection of $ W^{1,q}_{ \Gamma_{1}\cup \Gamma_{L}}(\Omega)$ into $L^{q_*}(\Omega)$. Thus, if $\Vert \tilde \theta_{\delta} \Vert_{1.q} \ge 1$ we have 
  \begin{eqnarray*}
 \Vert \tilde \theta_{\delta} \Vert_{1.q}^{\frac{q}{6-2q}} =  \Vert \tilde \theta_{\delta} \Vert_{1.q}^{1-\alpha} 
 \le C_1  2^{(q_*-1) \frac{2-q}{2q} } \left( \frac{|\Omega|^{\frac{2-q}{2q}} }{\Vert \tilde \theta_{\delta} \Vert_{1.q}^{\alpha}} + C_q^{\alpha}  \right) 
 \le C_1 2^{(q_*-1) \frac{2-q}{2q} } \left( |\Omega|^{\frac{2-q}{2q}} + C_q^{\alpha} \right)
 \end{eqnarray*}  
 and finally
   \begin{eqnarray*}
 \Vert \tilde \theta_{\delta} \Vert_{1.q} \le \max 
 \left( 1,  
  C_1^{\frac{2(3-q)}{q} } 2^{ \frac{(4q-3)(2-q)}{q^2} } \left( |\Omega|^{\frac{2-q}{2q}} + C_q^{\alpha} \right)^{\frac{2(3-q)}{q} } \right).
 \end{eqnarray*}
 \end{proof}

\begin{remark} Let us observe that we may prove an existence result for the approximate heat problem (\ref{heat_approx1})-(\ref{heat_approx1_bc}) by using a fixed point technique. Indeed, let $\tilde T: L^2(\Omega) \to L^2(\Omega)$ be defined by $\tilde T (\tilde \theta) = \tilde \theta_{\delta}$ where $\tilde \theta_{\delta} \in W^{1,2}_{ \Gamma_{1}\cup \Gamma_{L}}(\Omega)$ is the unique solution of (\ref{Lpb}).
Starting from (\ref{estim2}) we obtain immediately that $\tilde T \bigl( {\overline B}_{L^2(\Omega ) } (0, C'_{Poincare} R_{\delta} ) \bigr) \subset {\overline B}_{L^2(\Omega) } (0, C'_{Poincare} R_{\delta} )$ and $\tilde T \bigl( {\overline B}_{L^2(\Omega) } (0, C'_{Poincare} R_{\delta} ) \bigr)$ is relatively compact in $L^2(\Omega)$, where we recall that $C'_{Poincare}$ is the Poincar\'e's constant on $W^{1,2}_{ \Gamma_{1}\cup \Gamma_{L}}(\Omega)$. Then, by using the continuity properties and the uniform boundedness of the mappings $\tilde \theta \mapsto g_{\delta} (\tilde \theta, \upsilon)$ and $\tilde \theta \mapsto r(\tilde \theta)$ we can check the operator $\tilde T$ is continuous on $L^2(\Omega)$. Finally Schauder's fixed point theorem gives the existence of a solution to problem (\ref{heat_approx1})-(\ref{heat_approx1_bc}) as a fixed point of the operator $\tilde T$. Unfortunately this technique does not provide uniqueness and the non-linearity of the problem does not allow us to expect uniqueness unless we introduce further assumptions, like for instance some Lipschitz continuity properties for the mappings $\mu$ and $r$ with respect  to the temperature. 
\end{remark}

%%%%%%%%%%%%%%%%%%%%%%%%%%%%%%%%%%%%%%%%%%%%%%%%%%%%%%%%%%%%%%%
%            Debut de la section 5
%%%%%%%%%%%%%%%%%%%%%%%%%%%%%%%%%%%%%%%%%%%%%%%%%%%%%%%%%%%%%%%

% \renewcommand{\theequation}{5.\arabic{equation}}
% \setcounter{equation}{0}
\section{Existence result for the  coupled problem (P)} \label{section5}

We study in this section the  coupled problem (P). We proceed in two steps. First we will prove the existence of an approximate coupled problem depending on the parameter $\delta$ where the heat problem (\ref{TEM1})-(\ref{TEM3}) is replaced by (\ref{heat_approx1})-(\ref{heat_approx1_bc}) and then we will pass to the limit as $\delta$ tends to zero in order to obtain a solution to problem (P).

More precisely we prove first

\begin{proposition} \label{prop4}
%Let $\displaystyle q \in \bigl[1, 3/2 \bigr)$ and $p>3$ if $q=1$, $\displaystyle p \ge \frac{3q}{4q-3}$ otherwise
Let $p \ge 3/2$ 
and assume that (\ref{eqG})-(\ref{eqfk}), (\ref{rop})-(\ref{m5})-(\ref{mlo}) and (\ref{TEM2})-(\ref{TEM2bis})-(\ref{Cr})-(\ref{thetab}) hold. Then for any $\delta>0$ there exists $(\tilde \theta_{\delta}, \tilde \upsilon_{\delta}) \in W^{1,2}_{ \Gamma_{1}\cup \Gamma_{L}}(\Omega) \times V_{0.div}^{p}$ such that 
\begin{eqnarray} \label{flow_m}
\displaystyle
a(\tilde \theta_{\delta}; \tilde \upsilon_{\delta} , \varphi- \tilde  \upsilon_{\delta})
+\Psi(\varphi)-\Psi(\tilde  \upsilon_{\delta}) 
\geq \int_{\Omega} f\cdot (\varphi- \tilde  \upsilon_{\delta})\,dx 
\quad \forall \varphi\in V_{0.div}^{p}
\end{eqnarray}
and
\begin{eqnarray}  \label{heat_m}
B(\tilde  \upsilon_{\delta}, \tilde  \theta_{\delta}, w) = L_{\delta} (\tilde  \upsilon_{\delta}, \tilde  \theta_{\delta}, w) \quad \forall w \in W^{1,2}_{ \Gamma_{1}\cup \Gamma_{L}}(\Omega). 
%\hbox{\rm for all $w \in {\cal D}( {\overline{\Omega}})$  such that   $w=0$ on $ \Gamma_1 \cup \Gamma_L$.}
\end{eqnarray}
\end{proposition}

\begin{proof}
We define the operator $T: L^2(\Omega) \times {\bf L}^p(\Omega)  \to L^2(\Omega) \times {\bf L}^p(\Omega) $ by 
$T (\tilde \theta, \tilde \upsilon ) = ( \tilde \theta_{\delta}, \tilde \upsilon_{\delta}) $ where 
$\tilde \upsilon_{\delta} \in V^p_{0.div}$ is the unique solution of problem (\ref{flow_aux}) with $\upsilon = \tilde \upsilon$ and $\theta = \tilde \theta$ i.e.
\begin{eqnarray*}
\bigl\langle \mathbf{A}_{\tilde \theta, \tilde \upsilon} (\tilde \upsilon_{\delta}) , {\varphi} - \tilde \upsilon_{\delta} \bigr\rangle
+\Psi(\varphi)-\Psi(\tilde \upsilon_{\delta}) 
\geq \int_{\Omega} f\cdot (\varphi- \tilde \upsilon_{\delta})\,dx 
\quad \forall \varphi\in V_{0.div}^{p}
\end{eqnarray*}
and 
$\tilde \theta_{\delta} \in W^{1,2}_{ \Gamma_{1}\cup \Gamma_{L}}(\Omega)$ is the unique solution of (\ref{Lpb}) with $\upsilon = \tilde \upsilon_{\delta}$ i.e.
\begin{eqnarray*}
B(\tilde \upsilon_{\delta}, \tilde \theta_{\delta},w)=L_{\delta} (\tilde \upsilon_{\delta}, \tilde \theta ,w) \quad \forall w\in W^{1,2}_{ \Gamma_{1}\cup \Gamma_{L}}(\Omega).
\end{eqnarray*}

We already know with (\ref{estim2}) that
\begin{eqnarray*}
\begin{array}{ll}
\displaystyle \Vert \tilde \theta_{\delta} \Vert_{L^2(\Omega)} \le C'_{Poincare} \Vert \tilde \theta_{\delta} \Vert_{1.2} 
\\ 
\displaystyle
 \le C'_{Poincare} R_{\delta} := \frac{C'_{Poincare}}{k_0} \left( C'_{Poincare} \left( \frac{1}{\delta} + \| r\|_{L^{\infty} (\mathbb{R})} \right) |\Omega|^{1/2} +  \frac{C_{\gamma}}{\delta} |\Gamma_0|^{1/2}
%\| \theta^b_{\delta} \|_{L^2(\Gamma_0)} 
\right)
\end{array}
\end{eqnarray*}
where we recall that $C'_{Poincare}$ denotes the Poincar\'e's constant on $W^{1,2}_{ \Gamma_{1}\cup \Gamma_{L}}(\Omega)$ and from Proposition \ref{prop1} that
\begin{eqnarray*}
\Vert \tilde \upsilon_{\delta} \Vert_{{\bf L}^p(\Omega)} \le C_{Poincare} \Vert \tilde \upsilon_{\delta} \Vert_{1.p} \le C_{Poincare} {\mathcal C}_{flow}
\end{eqnarray*}
where $C_{Poincare}$ is the Poincar\'e's constant on $W^{1,p}_{ \Gamma_{1}\cup \Gamma_{L}}(\Omega)$. 

\noindent We infer that $ T \bigl( {\overline B}_{L^2(\Omega ) } (0, C'_{Poincare} R_{\delta} ) \times {\overline B}_{{\bf L}^p(\Omega ) } (0, C_{Poincare} {\mathcal C}_{flow} ) \bigr) $ is included into $ {\overline B}_{L^2(\Omega ) } (0, C'_{Poincare} R_{\delta} ) \times {\overline B}_{{\bf L}^p(\Omega ) } (0, C_{Poincare} {\mathcal C}_{flow} )$ and $ T \bigl( {\overline B}_{L^2(\Omega ) } (0, C'_{Poincare} R_{\delta} ) \times {\overline B}_{{\bf L}^p(\Omega ) } (0, C_{Poincare} {\mathcal C}_{flow} ) \bigr)$ is relatively compact in $L^2(\Omega) \times {\bf L}^p (\Omega) $.

\smallskip

Let us prove now that $T$ is continuous on $L^2(\Omega) \times {\bf L}^p(\Omega )$.
So let  $(\tilde\theta_{n}, \tilde \upsilon_n)_{n\geq 0}$  be 
a sequence of $L^2(\Omega) \times {\bf L}^p(\Omega )$  which  converges strongly to  $(\tilde \theta, \tilde \upsilon)$.
For all $n \ge 0$ we define $ ( \tilde \theta_{\delta}^n, \tilde \upsilon_{\delta}^n) = T (\tilde \theta_n, \tilde \upsilon_n )  $ and we have to prove that $( \tilde \theta_{\delta}^n, \tilde \upsilon_{\delta}^n)_{n \ge 0}$ converges strongly in $L^2(\Omega) \times {\bf L}^p (\Omega)$ to $( \tilde \theta_{\delta}, \tilde \upsilon_{\delta}) = T (\tilde \theta, \tilde \upsilon )  $.

For all $n \ge 0$ we have 
\begin{eqnarray*}
\bigl\langle \mathbf{A}_{\tilde \theta_n, \tilde \upsilon_n} (\tilde \upsilon_{\delta}^n) , {\varphi} - \tilde \upsilon_{\delta}^n \bigr\rangle
+\Psi(\varphi)-\Psi(\tilde \upsilon_{\delta}^n) 
\geq \int_{\Omega} f\cdot (\varphi- \tilde \upsilon_{\delta}^n)\,dx 
\quad \forall \varphi\in V_{0.div}^{p}
\end{eqnarray*}
and we have also
\begin{eqnarray*}
B(\tilde \upsilon_{\delta}^n, \tilde \theta_{\delta}^n,w)=L_{\delta} (\tilde \upsilon_{\delta}^n, \tilde \theta_n ,w) \quad \forall w\in W^{1,2}_{ \Gamma_{1}\cup \Gamma_{L}}(\Omega).
\end{eqnarray*}
%%%%%%%%%%%%%%%%%%%%%%%%%%%%%%%%%%%%%%%%%%%%%%%%%%%%%%%%%%%%%%%%%%%%%%%%%%%%%%%%
By choosing $\varphi = \tilde \upsilon_{\delta}$ then $\varphi = \tilde \upsilon_{\delta}^n$ we get 
\begin{eqnarray*} 
\begin{array}{ll}
\displaystyle 0 \ge \bigl\langle \mathbf{A}_{\tilde \theta_{n}, \tilde \upsilon_{n}} (\tilde \upsilon_{\delta}^n) - \mathbf{A}_{\tilde \theta, \tilde \upsilon} (\tilde \upsilon_{\delta}), \tilde \upsilon_{\delta}^n - \tilde \upsilon_{\delta}  \bigr\rangle \\
\displaystyle
= \mu_0 \int_{\Omega} \bigl( \bigl|D(  \tilde \upsilon_{\delta}^n +G) \bigr|^{p-2} D( \tilde \upsilon_{\delta}^n + G) 
- \bigl|D( \tilde \upsilon_{\delta} +G) \bigr|^{p-2} D(\tilde \upsilon_{\delta} + G) \Bigr) : D( \tilde \upsilon_{\delta}^n - \tilde \upsilon_{\delta}) \, dx \\
\displaystyle 
+ \int_{\Omega} \Bigl( {\overline {\mathcal F}} \bigl( \tilde \theta_{n}, \tilde \upsilon_n + G, D(\tilde \upsilon_{\delta}^n  +G) \bigr)
- {\overline {\mathcal F}} \bigl(\tilde  \theta, \tilde \upsilon +G, D(\tilde \upsilon_{\delta}  +G) \bigr) \Bigr)  : D(  \tilde \upsilon_{\delta}^n - \tilde \upsilon_{\delta}) \, dx \\
\displaystyle 
=  \mu_0 \int_{\Omega} \bigl( \bigl|D( \tilde \upsilon_{\delta}^n +G) \bigr|^{p-2} D( \tilde \upsilon_{\delta}^n + G) 
- \bigl|D( \tilde \upsilon_{\delta} +G) \bigr|^{p-2} D(\tilde \upsilon_{\delta} + G) \Bigr) : D( \tilde \upsilon_{\delta}^n - \tilde \upsilon_{\delta}) \, dx \\
\displaystyle 
%\mu_0 \left(\frac{1}{2} \right)^{p-1} \Vert  D(  \upsilon_m - \tilde \upsilon) \Vert^p_{L^p( \Omega)} \\
\displaystyle 
+ \int_{\Omega} \Bigl( {\overline {\mathcal F}} \bigl( \tilde \theta_{n}, \tilde \upsilon_n +G, D( \tilde \upsilon_{\delta}^n  +G) \bigr)
- {\overline {\mathcal F}} \bigl(\tilde  \theta_{n}, \tilde \upsilon_n +G, D(\tilde \upsilon_{\delta}  +G) \bigr) \Bigr)  : D( \tilde \upsilon_{\delta}^n - \tilde \upsilon_{\delta}) \, dx \\
\displaystyle
+ \int_{\Omega} \Bigl( {\overline {\mathcal F}} \bigl( \tilde \theta_{n}, \tilde \upsilon_n +G, D( \tilde \upsilon_{\delta}  +G) \bigr)
- {\overline {\mathcal F}} \bigl(\tilde  \theta, \tilde \upsilon +G, D(\tilde \upsilon_{\delta}  +G) \bigr) \Bigr)  : D(  \tilde \upsilon_{\delta}^n - \tilde \upsilon_{\delta}) \, dx
\end{array}
\end{eqnarray*}
and thus
\begin{eqnarray*} 
\begin{array}{ll}
\displaystyle \mu_0 \int_{\Omega} \bigl( \bigl|D( \tilde \upsilon_{\delta}^n +G) \bigr|^{p-2} D( \tilde \upsilon_{\delta}^n + G) 
- \bigl|D( \tilde \upsilon_{\delta} +G) \bigr|^{p-2} D(\tilde \upsilon_{\delta} + G) \Bigr) : D( \tilde \upsilon_{\delta}^n - \tilde \upsilon_{\delta}) \, dx \\
\displaystyle 
%\displaystyle \mu_0 \left(\frac{1}{2} \right)^{p-1} \Vert  D(  \upsilon_m - \tilde \upsilon) \Vert^p_{L^p( \Omega)}
\le \int_{\Omega} \Bigl( {\overline {\mathcal F}} \bigl( \tilde \theta, \tilde \upsilon +G, D( \tilde \upsilon_{\delta}  +G) \bigr)
- {\overline {\mathcal F}} \bigl( \tilde \theta_{n}, \tilde \upsilon_n +G , D( \tilde \upsilon_{\delta}  +G) \bigr) \Bigr)  : D(  \tilde \upsilon_{\delta}^n - \tilde \upsilon_{\delta}) \, dx \\
\displaystyle 
\le  \Bigl\| {\overline {\mathcal F}} \bigl(\tilde \theta, \tilde \upsilon + G, D( \tilde \upsilon_{\delta}  +G) \bigr)
- {\overline {\mathcal F}} \bigl( \tilde  \theta_{n}, \tilde \upsilon_n + G, D(\tilde \upsilon_{\delta}  +G) \bigr) \Bigr\|_{( L^{p'}(\Omega))^{3 \times 3}} 
 \Vert D( \tilde \upsilon_{\delta}^n - \tilde \upsilon_{\delta}) \Vert_{( L^p(\Omega))^{3 \times 3}}.
\end{array}
\end{eqnarray*}
Let us distinguish  two cases.

\smallskip

{\bf Case 1:} $p \ge 2$.

\smallskip

With (\ref{ineg_case1}) we get
\begin{eqnarray*}
\begin{array}{ll}
\displaystyle \mu_0 \frac{1}{2^{p-1}}  \Vert  D( \tilde \upsilon_{\delta}^n - \tilde \upsilon_{\delta}) \Vert^p_{(L^p( \Omega))^{3 \times 3}} \\
\displaystyle 
\le 
 \mu_0 \int_{\Omega} \bigl( \bigl|D( \tilde \upsilon_{\delta}^n +G) \bigr|^{p-2} D( \tilde \upsilon_{\delta}^n + G) 
- \bigl|D( \tilde \upsilon_{\delta} +G) \bigr|^{p-2} D(\tilde \upsilon_{\delta} + G) \Bigr) : D( \tilde \upsilon_{\delta}^n  - \tilde \upsilon_{\delta} ) \, dx \\
\displaystyle 
\le  \Bigl\| {\overline {\mathcal F}} \bigl(\tilde \theta, \tilde \upsilon + G, D( \tilde \upsilon_{\delta}  +G) \bigr)
- {\overline {\mathcal F}} \bigl( \tilde \theta_{n}, \tilde \upsilon_n + G, D( \tilde \upsilon_{\delta}  +G) \bigr) \Bigr\|_{ (L^{p'}(\Omega))^{3 \times 3}} 
 \Vert D( \tilde  \upsilon_{\delta}^n  - \tilde \upsilon_{\delta}) \Vert_{( L^p(\Omega))^{3 \times 3}}
\end{array}
\end{eqnarray*}
and thus
\begin{eqnarray*}
\displaystyle \mu_0 \frac{1}{2^{p-1}}  \Vert  D( \tilde  \upsilon_{\delta}^n - \tilde \upsilon_{\delta}) \Vert^{p-1}_{(L^p( \Omega))^{3 \times 3}}
\le \Bigl\| {\overline {\mathcal F}} \bigl( \tilde \theta, \tilde \upsilon +G , D( \tilde \upsilon_{\delta}  +G) \bigr)
- {\overline {\mathcal F}} \bigl( \tilde \theta_{n}, \tilde \upsilon_n + G, D( \tilde \upsilon_{\delta}  +G) \bigr) \Bigr\|_{( L^{p'}(\Omega))^{3 \times 3}} .
\end{eqnarray*}
By possibly extracting a subsequence, still denoted $(\tilde \theta_n, \tilde \upsilon_n)_{n  \ge 0}$, we have 
\begin{eqnarray*}
\tilde \theta_n \to \tilde \theta , \quad \tilde \upsilon_n \to \tilde \upsilon \quad \hbox{\rm a.e. in $\Omega$.}
\end{eqnarray*}
Then the continuity and uniform boundedness of the mapping ${\overline \mu}$ combined with Lebesgue's theorem imply that 
\begin{eqnarray*}
{\overline {\mathcal F}} \bigl( \tilde \theta_{n}, \tilde \upsilon_n + G, D( \tilde \upsilon_{\delta}  +G) \bigr) \Bigr) 
\to {\overline {\mathcal F}} \bigl( \tilde \theta, \tilde \upsilon +G , D(\tilde  \upsilon_{\delta}  +G) \bigr) \Bigr) \quad \hbox{\rm strongly in $(L^{p'}(\Omega))^{3 \times 3}$}.
\end{eqnarray*}
Thus 
\begin{eqnarray*}
\tilde \upsilon_{\delta}^n \to \tilde \upsilon_{\delta} \quad \hbox{\rm strongly in $V_{0.div}^p$.}
\end{eqnarray*}
Recalling that $(\tilde \upsilon_{\delta}^n)_{n \ge 0}$ is bounded in $V^p_{0.div}$, we may conclude that the whole sequence $(\tilde \upsilon_{\delta}^n)_{n \ge 0}$ converges strongly to $\tilde \upsilon_{\delta}$ in $V^p_{0.div}$.

\smallskip

\noindent {\bf Case 2:} $1<p<2$.

\smallskip

With (\ref{inegp}) and the same computations as in Theorem \ref{th1} we have
\begin{eqnarray*}
\begin{array}{ll}
\displaystyle C_p \int_{\Omega} \bigl| D( \tilde  \upsilon_{\delta}^n - \tilde \upsilon_{\delta}) \bigr|^p \, dx 
\\
\displaystyle
\le \left( \int_{\Omega} \bigl( \bigl|D( \tilde  \upsilon_{\delta}^n +G) \bigr|^{p-2} D( \tilde \upsilon_{\delta}^n + G) 
- \bigl|D( \tilde \upsilon_{\delta} +G) \bigr|^{p-2} D(\tilde \upsilon_{\delta} + G) \Bigr) : D(  \tilde \upsilon_{\delta}^n - \tilde \upsilon_{\delta}) \, dx \right)^{\frac{p}{2}} \\
\displaystyle 
\qquad  \times \left( \int_{\Omega} \Bigl(  \bigl| D( \tilde  \upsilon_{\delta}^n + G ) \bigr|^p +  \bigl| D( \tilde \upsilon_{\delta} + G ) \bigr|^p \Bigr) \, dx \right)^{\frac{2-p}{2}}
\end{array}
\end{eqnarray*}
and with Proposition \ref{prop1} we obtain
\begin{eqnarray*}
\begin{array}{ll}
\displaystyle  \frac{ C_p}{ ( 2^p ( C_{flow}^p + \Vert D(G) \Vert^p_{ (L^p(\Omega))^{3 \times 3} } ))^{\frac{2-p}{2}}}  \int_{\Omega} \bigl| D( \tilde  \upsilon_{\delta}^n - \tilde \upsilon_{\delta}) \bigr|^p \, dx \\
\displaystyle 
\le \left( \int_{\Omega} \bigl( \bigl|D( \tilde \upsilon_{\delta}^n +G) \bigr|^{p-2} D( \tilde \upsilon_{\delta}^n + G) 
- \bigl|D( \tilde \upsilon_{\delta} +G) \bigr|^{p-2} D(\tilde \upsilon_{\delta} + G) \Bigr) : D( \tilde  \upsilon_{\delta}^n - \tilde \upsilon_{\delta}) \, dx \right)^{\frac{p}{2}} .
\end{array}
\end{eqnarray*}
Thus
\begin{eqnarray*}
\begin{array}{ll}
\displaystyle \left( \frac{ C_p}{(2^p (C_{flow}^p + \Vert D(G) \Vert^p_{(L^p(\Omega))^{3 \times 3}} ))^{\frac{2-p}{2}}} \right)^{\frac{2}{p}}   \Vert D( \tilde  \upsilon_{\delta}^n - \tilde \upsilon_{\delta}) \Vert_{( L^p(\Omega))^{3 \times 3}} 
\\
\displaystyle 
\le \frac{1}{\mu_0} \Bigl\| {\overline {\mathcal F}} \bigl( \tilde \theta, \tilde \upsilon, D( \tilde \upsilon_{\delta}  +G) \bigr)
- {\overline {\mathcal F}} \bigl( \tilde \theta_{n}, \tilde \upsilon_n, D(\tilde  \upsilon_{\delta}  +G) \bigr) \Bigr\|_{( L^{p'}(\Omega))^{3 \times 3}}
\end{array}
\end{eqnarray*}
and we may conclude as in the previous case.
%%%%%%%%%%%%%%%%%%%%%%%%%%%%%%%%%%%%%%%%%%%%%%%%%%%%%%%%%%%%%%%

\smallskip

Next for all $n \ge 0$ we have 
\begin{eqnarray} \label{point-fixe}
B(\tilde \upsilon_{\delta}^n, \tilde \theta_{\delta}^n,w)=L_{\delta} (\tilde \upsilon_{\delta}^n, \tilde \theta_n ,w) \quad \forall w\in W^{1,2}_{ \Gamma_{1}\cup \Gamma_{L}}(\Omega).
\end{eqnarray}
We already know that, possibly extracting a subsequence, still denoted $(\tilde \theta_n, \tilde \upsilon_n)_{n  \ge 0}$, we have 
\begin{eqnarray*}
\tilde \theta_n \to \tilde \theta , \quad \tilde \upsilon_n \to \tilde \upsilon \quad \hbox{\rm a.e. in $\Omega$.}
\end{eqnarray*}
With the previous result we infer that, possibly extracting another subsequence, we have
\begin{eqnarray*}
\tilde \upsilon_{\delta}^n + G \to \tilde \upsilon_{\delta} +G, \quad D(\tilde \upsilon_{\delta}^n +G) \to D( \tilde \upsilon_{\delta} +G) \quad \hbox{\rm a.e. in $\Omega$.}
\end{eqnarray*}
By using the continuity and the uniform boundedness of the mappings $\mu$ and $r$ combined with Lebesgue's theorem we obtain that 
\begin{eqnarray*}
g_{\delta} ( \tilde \theta_n,  \tilde \upsilon_{\delta}^n) \to g_{\delta} ( \tilde \theta, \tilde \upsilon_{\delta}) \quad \hbox{\rm strongly in $L^2(\Omega)$}
\end{eqnarray*}
and 
\begin{eqnarray*}
r( \tilde \theta_n) \to r(\tilde \theta) \quad \hbox{\rm strongly in $L^2(\Omega)$.}
\end{eqnarray*}
On the other hand the sequence $(\tilde \theta_{\delta}^n)_{n \ge 0}$ is bounded in $W^{1,2}_{ \Gamma_{1}\cup \Gamma_{L}}(\Omega)$, so possibly extracting another subsequence there exists $\tilde \theta_{\delta}^* \in W^{1,2}_{ \Gamma_{1}\cup \Gamma_{L}}(\Omega)$ such that 
\begin{eqnarray*}
\tilde \theta_{\delta}^n \to \tilde \theta_{\delta}^* \quad \hbox{\rm weakly in $W^{1,2}_{ \Gamma_{1}\cup \Gamma_{L}}(\Omega)$ and strongly in $L^2(\Omega)$.}
\end{eqnarray*}
By passing to the limit in (\ref{point-fixe}) as $n$ tends to $+ \infty$ we obtain that
\begin{eqnarray*} 
B(\tilde \upsilon_{\delta}, \tilde \theta_{\delta}^*,w)=L_{\delta} (\tilde \upsilon_{\delta}, \tilde \theta ,w) \quad \forall w\in W^{1,2}_{ \Gamma_{1}\cup \Gamma_{L}}(\Omega).
\end{eqnarray*}
By uniqueness of the solution of (\ref{Lpb}) with $\upsilon = \tilde \upsilon_{\delta}$ we infer that $\tilde \theta_{\delta}^* = \tilde \theta_{\delta}$ and the whole sequence $(\tilde \theta_{\delta}^n)_{n \ge 0}$ converges strongly in $L^2(\Omega)$ to $\tilde \theta_{\delta}$.

\smallskip

Finally Schauder's fixed point theorem allows us to conclude.

\end{proof}

%%%%%%%%%%%%%%%%%%%%%%%%%%%%%%%%%%%%%%%%%%%%%%%%%%%%%%%%%%%%%%%%%%%%%%%%%%%%%%%%%%%%%%%%%%%%%%%%%

Let us consider now the sequence $(\theta_m, \upsilon_m)_{m \ge 1}$  defined by $(\theta_m, \upsilon_m) = (\tilde \theta_{1/m}, \tilde \upsilon_{1/m})$ for all $m \ge 1$ i.e.
\begin{eqnarray} \label{flow_m_bis}
\displaystyle
a(\tilde \theta_{1/m}; \tilde \upsilon_{1/m} , \varphi- \tilde  \upsilon_{1/m})
+\Psi(\varphi)-\Psi(\tilde  \upsilon_{1/m}) 
\geq \int_{\Omega} f\cdot (\varphi- \tilde  \upsilon_{1/m})\,dx 
\quad \forall \varphi\in V_{0.div}^{p}
\end{eqnarray}
and
\begin{eqnarray}  \label{heat_m_bis}
B(\tilde  \upsilon_{1/m}, \tilde  \theta_{1/m}, w) = L_{1/m} (\tilde  \upsilon_{1/m}, \tilde  \theta_{1/m}, w) \quad \forall w \in W^{1,2}_{ \Gamma_{1}\cup \Gamma_{L}}(\Omega). 
%\hbox{\rm for all $w \in {\cal D}( {\overline{\Omega}})$  such that   $w=0$ on $ \Gamma_1 \cup \Gamma_L$.}
\end{eqnarray}

Let $\displaystyle q \in \bigl[1, 3/2 \bigr)$ and $p>3$ if $q=1$, $\displaystyle p \ge \frac{3q}{4q-3}$ otherwise.
With the results of Theorem \ref{th1}, Proposition \ref{prop1} and Proposition \ref{prop3} we know that the sequences $(\upsilon_m)_{m \ge 1}$ and $(\theta_m)_{m \ge 1}$ are bounded in $ V_{0.div}^{p}$ and in $ W^{1,q}_{ \Gamma_{1}\cup \Gamma_{L}}(\Omega)$ respectively. Reminding that we have to  deal with any $\displaystyle q \in \bigl[1, 3/2 \bigr)$ we can not infer directly any good convergence property for $(\nabla \theta_m)_{m \ge 1}$ since the unit ball of $ W^{1,q}_{ \Gamma_{1}\cup \Gamma_{L}}(\Omega)$ is not weakly compact when $q=1$. So we let $\displaystyle \zeta \in \left( 0, \frac{3-2q}{3-q} \right)$. Then $\displaystyle \rho= \frac{(\zeta +1) q}{2-q} \in (1, q_*)$ and $W^{1,q}_{ \Gamma_{1}\cup \Gamma_{L}}(\Omega)$ is compactly embedded into $L^{\rho}(\Omega)$. 
It follows that, possibly extracting a subsequence, still denoted $(\upsilon_m)_{m \ge 1}$ and $(\theta_m)_{m \ge 1}$, there exist $\upsilon \in  V_{0.div}^{p}$ and   $\theta \in L^{\rho}(\Omega)$ such that
\begin{eqnarray} \label{conv1}
\upsilon_m \to \upsilon \quad \hbox{\rm weakly in $V_{0.div}^{p}$ and strongly in ${\bf L}^p(\Omega)$} 
\end{eqnarray}
and
\begin{eqnarray} \label{conv2}
\theta_m \to \theta \quad \hbox{\rm 
%weakly in $W^{1,q}_{ \Gamma_{1}\cup \Gamma_{L}}(\Omega)$ and 
strongly in $L^{\rho}(\Omega)$.}
\end{eqnarray}
Moreover, possibly extracting another subsequence, we have
\begin{eqnarray} \label{conv3}
\upsilon_m \to \upsilon, \quad \theta_m \to \theta \quad \hbox{\rm a.e. in $\Omega$.}
\end{eqnarray}
%This is not enough to pass to the limit in (\ref{flow_m})-(\ref{heat_m}) and we need also 
Let us prove now the strong convergence of $\bigl( D(\upsilon_m + G) \bigr)_{m \ge 1}$ in ${\bf L}^p(\Omega)$. 

\begin{proposition} \label{prop4bis}
Under the previous assumptions 
\begin{eqnarray} \label{conv4}
\upsilon_m \to \upsilon \quad \hbox{\rm strongly in $V_{0.div}^{p}$}
\end{eqnarray}
and $\upsilon$ satisfies the variational inequality
\begin{eqnarray} \label{limit-flow-pb}
\displaystyle
a(\theta ;\upsilon , \varphi-\upsilon)
+\Psi(\varphi)-\Psi(\upsilon) 
\geq \int_{\Omega} f\cdot (\varphi-\upsilon)\,dx 
\quad \forall \varphi\in V_{0.div}^{p}.
\end{eqnarray}
\end{proposition}

\begin{proof}
Let $\tilde \upsilon \in V_{0.div}^{p}$ be the unique solution of 
 the following auxiliary problem
\begin{eqnarray} \label{flow_aux_bis}
%\left\{
%\begin{array}{ll}
%\displaystyle \hbox{\rm Find $\tilde \upsilon \in V_{0.div}^{p}$ such that} \\
%\displaystyle 
\bigl\langle \mathbf{A}_{\theta, \upsilon} (\tilde \upsilon) , {\varphi} - \tilde \upsilon \bigr\rangle
+\Psi(\varphi)-\Psi(\tilde \upsilon) 
\geq \int_{\Omega} f\cdot (\varphi- \tilde \upsilon)\,dx 
\quad \forall \varphi\in V_{0.div}^{p} .
%\end{array}
%\right.
\end{eqnarray}

%%%%%%%%%%%%%%%%%%%%%%%%%%%%%%%%%%%
%%%%%%%%%%%%%%%%%%%%%%%%%%%%%%%%%%%%

Now we let $\varphi = \upsilon_m $ in (\ref{flow_aux_bis}) and $\varphi = \tilde \upsilon$ in (\ref{flow_m_bis}). We obtain
\begin{eqnarray*} 
\begin{array}{ll}
\displaystyle 0 \ge \bigl\langle \mathbf{A}_{\theta_{m}, \upsilon_m} (\upsilon_m) - \mathbf{A}_{\theta, \upsilon} (\tilde \upsilon),  \upsilon_m - \tilde \upsilon  \bigr\rangle \\
\displaystyle
= \mu_0 \int_{\Omega} \bigl( \bigl|D(  \upsilon_m +G) \bigr|^{p-2} D( \upsilon_m+ G) 
- \bigl|D( \tilde \upsilon +G) \bigr|^{p-2} D(\tilde \upsilon+ G) \Bigr) : D(  \upsilon_m - \tilde \upsilon) \, dx \\
\displaystyle 
+ \int_{\Omega} \Bigl( {\overline {\mathcal F}} \bigl( \theta_{m}, \upsilon_m + G, D( \upsilon_m  +G) \bigr)
- {\overline {\mathcal F}} \bigl( \theta, \upsilon +G, D(\tilde \upsilon  +G) \bigr) \Bigr)  : D(  \upsilon_m - \tilde \upsilon) \, dx \\
\displaystyle 
=  \mu_0 \int_{\Omega} \bigl( \bigl|D(  \upsilon_m +G) \bigr|^{p-2} D( \upsilon_m+ G) 
- \bigl|D( \tilde \upsilon +G) \bigr|^{p-2} D(\tilde \upsilon+ G) \Bigr) : D(  \upsilon_m - \tilde \upsilon) \, dx \\
\displaystyle 
%\mu_0 \left(\frac{1}{2} \right)^{p-1} \Vert  D(  \upsilon_m - \tilde \upsilon) \Vert^p_{L^p( \Omega)} \\
\displaystyle 
+ \int_{\Omega} \Bigl( {\overline {\mathcal F}} \bigl( \theta_{m}, \upsilon_m +G, D( \upsilon_m  +G) \bigr)
- {\overline {\mathcal F}} \bigl( \theta_{m}, \upsilon_m +G, D(\tilde \upsilon  +G) \bigr) \Bigr)  : D(  \upsilon_m - \tilde \upsilon) \, dx \\
\displaystyle
+ \int_{\Omega} \Bigl( {\overline {\mathcal F}} \bigl( \theta_{m}, \upsilon_m +G, D( \tilde \upsilon  +G) \bigr)
- {\overline {\mathcal F}} \bigl( \theta, \upsilon +G, D(\tilde \upsilon  +G) \bigr) \Bigr)  : D(  \upsilon_m - \tilde \upsilon) \, dx
\end{array}
\end{eqnarray*}
and thus
\begin{eqnarray*} 
\begin{array}{ll}
\displaystyle  \mu_0 \int_{\Omega} \bigl( \bigl|D(  \upsilon_m +G) \bigr|^{p-2} D( \upsilon_m+ G) 
- \bigl|D( \tilde \upsilon +G) \bigr|^{p-2} D(\tilde \upsilon+ G) \Bigr) : D(  \upsilon_m - \tilde \upsilon) \, dx \\
\displaystyle 
%\displaystyle \mu_0 \left(\frac{1}{2} \right)^{p-1} \Vert  D(  \upsilon_m - \tilde \upsilon) \Vert^p_{L^p( \Omega)}
\le \int_{\Omega} \Bigl( {\overline {\mathcal F}} \bigl( \theta, \upsilon +G, D( \tilde \upsilon  +G) \bigr)
- {\overline {\mathcal F}} \bigl( \theta_{m}, \upsilon_m +G , D( \tilde \upsilon  +G) \bigr) \Bigr)  : D(  \upsilon_m - \tilde \upsilon) \, dx \\
\displaystyle 
\le  \Bigl\| {\overline {\mathcal F}} \bigl( \theta, \upsilon, D( \tilde \upsilon  +G) \bigr)
- {\overline {\mathcal F}} \bigl( \theta_{m}, \upsilon_m, D(\tilde  \upsilon  +G) \bigr) \Bigr\|_{ (L^{p'}(\Omega))^{3 \times 3}} 
 \Vert D(  \upsilon_m - \tilde \upsilon) \Vert_{( L^p(\Omega))^{3 \times 3}}.
\end{array}
\end{eqnarray*}
%for all $\lambda >0$. We choose $\displaystyle \lambda =  (\mu_0 p)^{\frac{1}{p-1}} \left( \frac{1}{2} \right)^{p'} >0$ and we get
%\begin{eqnarray*} 
%\begin{array}{ll}
%\displaystyle \mu_0 \left(\frac{1}{2} \right)^{p} \Vert  D(  \upsilon_m - \tilde \upsilon) \Vert^p_{L^p( \Omega)} \\
%\displaystyle 
%\le \frac{1}{\lambda p'} \Bigl\| {\overline {\cal F}} \bigl( \theta, \upsilon, D( \tilde \upsilon  +G) \bigr)
%- {\overline {\cal F}} \bigl( \theta_m, \upsilon_m, D( \upsilon  +G) \bigr) \Bigr\|^{p'}_{{\bf L}^{p'}(\Omega)}. 
%\end{array}
%\end{eqnarray*}
Let us distinguish once again two cases.

\smallskip

{\bf Case 1:} $p \ge 2$.

\smallskip

With the same computations as in Proposition \ref{prop4} we get
\begin{eqnarray*}
\begin{array}{ll}
\displaystyle \mu_0 \frac{1}{2^{p-1}}  \Vert  D(  \upsilon_m - \tilde \upsilon) \Vert^p_{(L^p( \Omega))^{3 \times 3}} 
\\
\displaystyle 
\le 
 \mu_0 \int_{\Omega} \bigl( \bigl|D(  \upsilon_m +G) \bigr|^{p-2} D( \upsilon_m+ G) 
- \bigl|D( \tilde \upsilon +G) \bigr|^{p-2} D(\tilde \upsilon+ G) \Bigr) : D(  \upsilon_m - \tilde \upsilon) \, dx \\
\displaystyle 
\le  \Bigl\| {\overline {\mathcal F}} \bigl( \theta, \upsilon, D( \tilde \upsilon  +G) \bigr)
- {\overline {\mathcal F}} \bigl( \theta_{m}, \upsilon_m, D(\tilde \upsilon  +G) \bigr) \Bigr\|_{( L^{p'}(\Omega))^{3 \times 3} } 
 \Vert D(  \upsilon_m - \tilde \upsilon) \Vert_{( L^p(\Omega))^{3 \times 3}}
\end{array}
\end{eqnarray*}
and thus
\begin{eqnarray*}
%\begin{array}{ll}
\displaystyle \mu_0  \frac{1}{2^{p-1} }  \Vert  D(  \upsilon_m - \tilde \upsilon) \Vert^{p-1}_{(L^p( \Omega))^{3 \times 3}} 
%\\
\displaystyle 
\le \Bigl\| {\overline {\mathcal F}} \bigl( \theta, \upsilon, D( \tilde \upsilon  +G) \bigr)
- {\overline {\mathcal F}} \bigl( \theta_{m}, \upsilon_m, D( \tilde \upsilon  +G) \bigr) \Bigr\|_{( L^{p'}(\Omega))^{3 \times 3}} .
%\end{array}
\end{eqnarray*}
By using the continuity and uniform boundedness of the mapping ${\overline \mu}$, the convergence (\ref{conv3}) and Lebesgue's theorem we obtain that 
\begin{eqnarray*}
{\overline {\mathcal F}} \bigl( \theta_{m}, \upsilon_m + G, D( \tilde \upsilon  +G) \bigr) \Bigr) 
\to {\overline {\mathcal F}} \bigl( \theta, \upsilon +G , D( \tilde \upsilon  +G) \bigr) \Bigr) \quad \hbox{\rm strongly in $\bigl(L^{p'}(\Omega) \bigr)^{3 \times 3}$}.
\end{eqnarray*}
Thus 
\begin{eqnarray*}
\upsilon_m \to \tilde \upsilon \quad \hbox{\rm strongly in $V_{0.div}^p$}
\end{eqnarray*}
and with (\ref{conv1}) we infer that $\upsilon = \tilde \upsilon$ which yields the announced result.

\smallskip

\noindent {\bf Case 2:} $1<p<2$.

\smallskip

%Then 
%\begin{eqnarray*}
%\begin{array}{ll}
%\displaystyle C_p \int_{\Omega} \bigl| D(  \upsilon_m - \tilde \upsilon) \bigr|^p \, dx 
%\le \left( \int_{\Omega} \bigl( \bigl|D(  \upsilon_m +G) \bigr|^{p-2} D( \upsilon_m+ G) 
%- \bigl|D( \tilde \upsilon +G) \bigr|^{p-2} D(\tilde \upsilon + G) \Bigr) : D(  \upsilon_m - \tilde \upsilon) \, dx \right)^{\frac{p}{2}} \\
%\displaystyle \times \left( \int_{\Omega} \Bigl(  \bigl| D(  \upsilon_m + G ) \bigr|^p +  \bigl| D( \tilde \upsilon + G ) \bigr|^p \Bigr) \, dx \right)^{\frac{2-p}{2}}
%\end{array}
%\end{eqnarray*}
%and with Proposition \ref{prop1} we obtain
%\begin{eqnarray*}
%\begin{array}{ll}
%\displaystyle  \frac{ C_p}{( 2^p ( C_{flow}^p + \Vert D(G) \Vert^p_{ (L^p(\Omega))^{3 \times 3} }  )^{\frac{2-p}{2}}}  \int_{\Omega} \bigl| D(  \upsilon_m - \tilde \upsilon) \bigr|^p \, dx 
%\le \left( \int_{\Omega} \bigl( \bigl|D(  \upsilon_m +G) \bigr|^{p-2} D( \upsilon_m+ G) 
%- \bigl|D( \tilde \upsilon +G) \bigr|^{p-2} D(\tilde \upsilon + G) \Bigr) : D(  \upsilon_m - \tilde \upsilon) \, dx \right)^{\frac{p}{2}} .
%\end{array}
%\end{eqnarray*}
Once again with the same computations as in Proposition \ref{prop4}
\begin{eqnarray*}
\begin{array}{ll}
\displaystyle \left( \frac{ C_p}{( 2^p ( C_{flow}^p + \Vert D(G) \Vert^p_{ (L^p(\Omega))^{3 \times 3} }  )^{\frac{2-p}{2}}} \right)^{\frac{2}{p}}   \Vert D(  \upsilon_m - \tilde \upsilon) \Vert_{ (L^p(\Omega))^{3 \times 3} }
 \\
\displaystyle 
\le \frac{1}{\mu_0} \Bigl\| {\overline {\mathcal F}} \bigl( \theta, \upsilon, D( \tilde \upsilon  +G) \bigr)
- {\overline {\mathcal F}} \bigl( \theta_{m}, \upsilon_m, D( \tilde \upsilon  +G) \bigr) \Bigr\|_{( L^{p'}(\Omega))^{3 \times 3}}.
\end{array}
\end{eqnarray*}
and we may conclude as in the previous case.

\end{proof}

We may  construct now  a pressure $\pi \in L^{p'}_0(\Omega)$ satisfying
\begin{eqnarray} \label{pressure}
\displaystyle
a(\theta;\upsilon , \varphi-\upsilon) - \int_{\Omega} \pi {\rm div}(\varphi) \, dx 
+\Psi(\varphi)-\Psi(\upsilon) 
\geq \int_{\Omega} f\cdot (\varphi-\upsilon )\,dx 
\quad \forall \varphi\in V_{0}^{p}.
\end{eqnarray}

As usual the main tool is De Rham's theorem.
 
 \begin{proposition} \label{prop5}
 Let $q \in \bigl[1, 3/2 \bigr)$ and $p>3$ if $q=1$, $\displaystyle p \ge \frac{3q}{4q-3}$ otherwise and assume that (\ref{eqG})-(\ref{eqfk}), (\ref{rop})-(\ref{m5})-(\ref{mlo}) and (\ref{TEM2})-(\ref{TEM2bis})-(\ref{Cr})-(\ref{thetab}) hold.
 Let $(\upsilon, \theta)$ be obtained as previously. Then there exists a unique $\pi \in L^{p'}_0(\Omega)$ satisfying (\ref{pressure}).
 \end{proposition}
 
\begin{proof}
We introduce the linear form $\mathcal{L}$  defined on $V_{0 }^p$ by 
\begin{eqnarray}\label{brux}
\mathcal{L}(\psi)=\int_{\Omega} f \cdot \psi \, dx  -a(\theta;\upsilon,\psi) \quad \forall \psi \in V_{0 }^p.
\end{eqnarray}
Clearly $\mathcal{L}$ is continuous on $V_{0 }^p$. 
Let  
\begin{eqnarray*}
\begin{array}{ll} 
 \displaystyle {\bf W}^{1,p}_{0}(\Omega) = \bigl\{ w \in {\bf W}^{1,p}(\Omega): \ w=0 \ {\rm on} \ \partial \Omega \bigr\}, \quad \\
 \displaystyle  {\bf W}^{1,p}_{0.div}(\Omega) = \bigl\{ w \in {\bf W}^{1,p}_{0}(\Omega): \ {\rm div} (w) = 0 \ {\rm in } \ \Omega \bigr\}.
 \end{array}
 \end{eqnarray*}
  By choosing $\varphi = \upsilon \pm \psi $ with $\psi \in {\bf W}^{1,p}_{0.div}(\Omega)$ as a test-function in (\ref{limit-flow-pb}) we get $\mathcal{L}(\psi) =0$ for all $\psi \in {\bf W}^{1,p}_{0.div}(\Omega)$.  With De Rham's theorem (see for instance Lemma   2.7 
% page 115 with  $m=1,r'=p,r={p'}$ 
in  \cite{Amrouche-Girault}) we infer that there 
exists  a unique   $\pi\in L^{p'}_0(\Omega)$ such that
\begin{eqnarray*}\label{halak}
\mathcal{L}(\psi)= \int_{\Omega} f \cdot \psi \, dx  -a(\theta;\upsilon,\psi) = \langle \nabla \pi, \psi \rangle_{\bf{\mathcal D}'(\Omega),\bf{\mathcal D}(\Omega)} 
\quad\forall\psi \in\bf{\mathcal D}(\Omega).
\end{eqnarray*}
On the other hand, with Green's formula
\begin{eqnarray*}
\begin{array}{ll}
\displaystyle  a(\theta;\upsilon,\psi) 
= \int _{\Omega}
2\mu \bigl(\theta,\upsilon+G,
|D(\upsilon+G)| \bigr) \bigl|D(\upsilon+G) \bigr|^{p-2} D(\upsilon+G):D(\psi) \, dx \\
 = - 
\Bigl\langle {\rm div} \bigl( 2\mu \bigl(\theta,\upsilon+G,|D(\upsilon+G)| \bigr) 
\bigl|D(\upsilon+G) \bigr|^{p-2} D( \upsilon+G) \bigl), 
\psi \Bigr\rangle_{\bf{\mathcal D}'(\Omega),\bf{\mathcal D}(\Omega)}
\end{array}
\end{eqnarray*}
and thus 
\begin{eqnarray} \label{pression1}
\begin{array}{ll}
\displaystyle 
- \Bigl\langle {\rm div} \bigl( 2\mu \bigl(\theta,\upsilon+G,|D(\upsilon+G)| \bigr) 
\bigl|D(\upsilon+G) \bigr|^{p-2} D( \upsilon+G) \bigl), 
\psi \Bigr\rangle_{\bf{\mathcal D}'(\Omega),\bf{\mathcal D}(\Omega)} 
\\
\displaystyle
+ \langle \nabla\pi,\psi \rangle_{\bf{\mathcal D}'(\Omega),\bf{\mathcal D}(\Omega)} 
= \int_{\Omega} f \cdot \psi \, dx
\end{array}
\end{eqnarray}
for all  $\psi \in\bf{\mathcal D}(\Omega)$. We define the stress tensor as (see (\ref{sigma}))
\begin{eqnarray} \label{pression2}
\sigma = 2 \mu \bigl( \theta, \upsilon +G, \bigl| D(\upsilon +G) \bigr| \bigr) \bigl| D(\upsilon +G) \bigr|^{p-2} D(\upsilon +G) - \pi {\rm Id}_{{\mathbb R}^3} .
\end{eqnarray}
With (\ref{pression1}) we have
\begin{eqnarray*}\label{ss"}
 -\bigl\langle {\rm div} ( \sigma), \psi \bigr\rangle_{\bf{\mathcal D}'(\Omega),\bf{\mathcal D}(\Omega)} 
=  \int_{\Omega} f \cdot \psi \, dx
\quad\forall\psi \in\bf{\mathcal D}(\Omega).   
\end{eqnarray*}
Recalling that $f \in {\bf L}^{p'}(\Omega)$ we obtain that ${\rm div} (\sigma) \in {\bf L}^{p'}(\Omega)$ and 
\begin{eqnarray*}
- {\rm div} (\sigma)  = f \quad \hbox{\rm in ${\bf L}^{p'}(\Omega)$.}
\end{eqnarray*}
With (\ref{pression2}) we have also $\sigma \in {\bf L}^{p'} (\Omega)$ and with Green's formula we get
\begin{eqnarray*}
\int_{\Omega} \sigma : \nabla w \, dx - \int_{\partial \Omega} \sum_{i,j =1}^3 \sigma_{ij} n_j w_i \, dY = \int_{\Omega} f \cdot w \, dx \quad \forall w \in {\bf W}^{1,p}(\Omega).
\end{eqnarray*}
where we recall that $n$ denotes the unit outward normal vector to $\partial \Omega$. We choose $w = \varphi - \upsilon$ with $\varphi \in V_0^p$. We get 
\begin{eqnarray*}
\int_{\Omega} \sigma : \nabla (\varphi - \upsilon)  \, dx - \int_{\Gamma_0} \sum_{i,j =1}^3 \sigma_{ij} n_j (\varphi - \upsilon)_i \, dx' = \int_{\Omega} f \cdot (\varphi - \upsilon) \, dx  \quad \forall \varphi \in V_0^p.
\end{eqnarray*}
We observe that the vector   $\sigma \cdot n = \left( \sum_{j=1}^3 \sigma_{ij}n_{j} \right)_{1 \le i \le 3}$ can be decomposed as $\sigma_{\tau}+\sigma_{n} n$. 
% with  $\sigma_n = \sum_{i,j =1}^3 \sigma_{ij} n_i n_j$. 
Since $(\varphi - \upsilon) \cdot n = 0$ on $\Gamma_0$ we obtain
\begin{eqnarray*}
\int_{\Omega} \sigma : \nabla (\varphi - \upsilon)  \, dx - \int_{\Gamma_0}  \sigma_{\tau} \cdot (\varphi - \upsilon) \, dx' = \int_{\Omega} f \cdot (\varphi - \upsilon) \, dx  \quad \forall \varphi \in V_0^p
\end{eqnarray*}
and by replacing $\sigma$ by its expression (see (\ref{pression2})) we get
\begin{eqnarray*}\label{eq329} %nadjm}
a(\theta; {\upsilon},\varphi - {\upsilon} 
)- \int_{\Omega} \pi {\rm div} (\varphi ) \, dx -
\int_{\Omega} f \cdot (\varphi - \upsilon) \, dx  
=\int_{\Gamma_{0}} 
\sigma_{\tau} \cdot (\varphi - \upsilon )\, dx'
\end{eqnarray*}
i.e.
\begin{eqnarray*}
\begin{array}{ll}
\displaystyle 
a(\theta; {\upsilon},\varphi - {\upsilon} 
)- \int_{\Omega} \pi {\rm div} (\varphi ) \, dx -
\int_{\Omega} f \cdot (\varphi - \upsilon) \, dx  
+ \Psi( \varphi) - \Psi(\upsilon) 
\\
\displaystyle 
=\int_{\Gamma_{0}} 
\sigma_{\tau} \cdot (\varphi - \upsilon )\, dx' + \Psi( \varphi) - \Psi(\upsilon)  \quad \forall \varphi \in V_0^p.
\end{array}
\end{eqnarray*}
On the other hand, we have
\[\int_{\partial\Omega} \varphi\cdot n\, dY=0 \quad\forall 
 \varphi\in V_{0}^p \]
so
there exists  $\hat  \varphi\in \bf{W}^{1,p}(\Omega)$ 
such that 
 ${\rm div} (\hat\varphi)=0 $ in $ \Omega$ and $ \hat \varphi=  
\varphi$ on $\partial\Omega$ (see for instance Corollary 3.8 in  \cite{Amrouche-Girault}). Hence $\hat \varphi \in V_{0.div}^p$ and with (\ref{limit-flow-pb})
\begin{eqnarray*}
\begin{array}{ll}
\displaystyle 
a(\theta; {\upsilon},\hat \varphi - {\upsilon} 
)- \underbrace{\int_{\Omega} \pi {\rm div} (\hat \varphi ) \, dx}_{=0} -
\int_{\Omega} f \cdot (\hat \varphi - \upsilon) \, dx  
+ \Psi(\hat \varphi) - \Psi(\upsilon) 
\\
\displaystyle 
=\int_{\Gamma_{0}} 
\sigma_{\tau} \cdot (\hat \varphi - \upsilon )\, dx' + \Psi(\hat \varphi) - \Psi(\upsilon) \ge 0.
\end{array}
\end{eqnarray*}
But 
\begin{eqnarray*}
\int_{\Gamma_{0}} 
\sigma_{\tau} \cdot (\varphi - \upsilon )\, dx' + \Psi( \varphi) - \Psi(\upsilon)  
= \int_{\Gamma_{0}} 
\sigma_{\tau} \cdot (\hat \varphi - \upsilon )\, dx' + \Psi(\hat \varphi) - \Psi(\upsilon) 
\end{eqnarray*}
since $ \hat \varphi=  \varphi$ on $\partial\Omega$. Thus
\begin{eqnarray*}
\begin{array}{ll}
\displaystyle 
a(\theta; {\upsilon},\varphi - {\upsilon} 
)- \int_{\Omega} \pi {\rm div} (\varphi ) \, dx -
\int_{\Omega} f \cdot (\varphi - \upsilon) \, dx  
+ \Psi( \varphi) - \Psi(\upsilon) 
\\
\displaystyle
=\int_{\Gamma_{0}} 
\sigma_{\tau} \cdot (\varphi - \upsilon )\, dx' + \Psi( \varphi) - \Psi(\upsilon) \ge 0 \quad \forall \varphi \in V_0^p
\end{array}
\end{eqnarray*}
and the variational inequality (\ref{pressure}) is satisfied.
\end{proof}

Finally 
%with Proposition \ref{prop5} we infer that the triplet $(\upsilon, \pi, \theta)$ is a solution to the coupled problem (P) and 
we can state the following existence result 

\begin{theorem} \label{main_theorem}
 Let $p>3/2$ and $\displaystyle q \in \bigl[1, 3/2 \bigr)$ satisfying the compatibility condition (\ref{compa1}) i.e.  $p>3$ if $q=1$, $\displaystyle p \ge \frac{3q}{4q-3}$ otherwise and assume that (\ref{eqG})-(\ref{eqfk}), (\ref{rop})-(\ref{m5})-(\ref{mlo}) and (\ref{TEM2})-(\ref{TEM2bis})-(\ref{Cr})-(\ref{thetab}) hold. Then the coupled problem (P) admits a solution.
 \end{theorem}
 
 \begin{proof}
 Let us prove that the triplet $(\upsilon, \pi, \theta)$ is a solution to the coupled problem (P) i.e. let us prove that 
\begin{eqnarray*}
\begin{array}{ll}
\displaystyle  \int _{\Omega} \bigl( (\upsilon+G) \cdot\nabla\theta \bigr)w\,dx
+ \int _{\Omega} (K \nabla \theta) \cdot \nabla w\,dx 
= \int _{\Omega}2 \mu \bigl(\theta , \upsilon +G, \bigl|D(\upsilon +G) \bigr| \bigr)
\bigl|D(\upsilon +G) \bigr|^p w \,dx 
 \\
\displaystyle 
+\int _{\Omega} r(\theta ) w \,dx + \int_{\Gamma_{0}} \theta^{b} w \,dx' \quad \forall w \in C^1 ({\overline  \Omega}) \ {\rm s.t.}  \ w = 0 \ {\rm on} \ \Gamma_1 \cup \Gamma_L.
\end{array}
\end{eqnarray*}

With Proposition \ref{prop4bis}, by extracting possibly another subsequence, we have 
\begin{eqnarray} \label{conv5}
D( \upsilon_m +G) \to D( \upsilon+G)   \quad \hbox{\rm a.e. in $\Omega$.}
\end{eqnarray}
By using the continuity and uniform boundedness of the mappings $\mu$ and $r$ combined with the convergence properties (\ref{conv3})-(\ref{conv4})-(\ref{conv5}) we obtain that
\begin{eqnarray} \label{conv6}
\displaystyle g_{1/m} ( \theta_m, \upsilon_m)   \to 
2 \mu \bigl(\theta, \upsilon +G,  \bigl| D( \upsilon +G) \bigr| \bigr) \bigl|D(\upsilon + G) \bigr|^p \quad \hbox{\rm strongly in $L^1(\Omega)$} 
\end{eqnarray}
and
\begin{eqnarray} \label{conv7}
\displaystyle r(\theta_m) \to r(\theta) \quad \hbox{\rm strongly in $L^1(\Omega)$}.
\end{eqnarray}
Moreover 
\begin{eqnarray} \label{conv8}
\displaystyle \theta^b_{1/m}  \to \theta^b \quad \hbox{\rm strongly in $L^1(\Gamma_0)$}.
\end{eqnarray}
These properties allow us to pass to the limit in the right hand side of (\ref{heat_m_bis}) for all $w \in C^1 ({\overline  \Omega})$ such that $w=0$ on $\Gamma_1 \cup \Gamma_L$.

\smallskip

As  explained at the beginning of Section \ref{section5} the uniform boundedness of the sequence $(\theta_m)_{m \ge 1}$  in $W^{1,q}_{\Gamma_1 \cup \Gamma_L} (\Omega)$ does not yield  any good convergence property for $(\nabla \theta_m)_{m \ge 1}$ when $q=1$ but we already know that the (sub-)sequence $(\theta_m)_{m \ge 1}$ converges strongly to $\theta$ in $L^{\rho}(\Omega)$ with $\displaystyle \rho = \frac{(\zeta+1) q}{2-q} \in (1, q_*)$ (see (\ref{conv2})). 

We may improve this result and we will show that $(\theta_m)_{m \ge 1}$ converges strongly to $\theta$ in $W^{1,q}_{\Gamma_1 \cup \Gamma_L} (\Omega)$. Indeed let $m_1 \ge 1$ and $m_2 \ge 1$. With (\ref{heat_m}) we have
\begin{eqnarray*}
B (\upsilon_{m_1}, \theta_{m_1}, w) - B (\upsilon_{m_2}, \theta_{m_2}, w) = L_{1/m_1} (\upsilon_{m_1}, \theta_{m_1}, w) - L_{1/m_2} (\upsilon_{m_2}, \theta_{m_2}, w) \quad \forall w \in W^{1,2}_{\Gamma_1 \cup \Gamma_L} (\Omega)
\end{eqnarray*}
i.e.
\begin{eqnarray*}
\begin{array}{ll}
\displaystyle  \int _{\Omega} \bigl( (\upsilon_{m_1}+G) \cdot\nabla (\theta_{m_1} - \theta_{m_2} ) \bigr) w\,dx
+ \int _{\Omega} \bigl( K \nabla (\theta_{m_1} - \theta_{m_2}) \bigr) \cdot \nabla w\,dx 
=  \int _{\Omega} \bigl( (\upsilon_{m_2} - \upsilon_{m_1} ) \cdot\nabla \theta_{m_2} \bigr) w\,dx
 \\
\displaystyle 
+ \int _{\Omega} \bigl( g_{1/m_1} ( \theta_{m_1}, \upsilon_{m_1}) - g_{1/m_2} ( \theta_{m_2}, \upsilon_{m_2}) \bigr) w \, dx  
+\int _{\Omega} \bigl(  r(\theta_{m_1} ) - r(\theta_{m_2} ) \bigr)  w \,dx 
\\
\displaystyle 
\ \ + \int_{\Gamma_{0}} \bigl( \theta^{b}_{1/m_1} - \theta^{b}_{1/m_2} \bigr) w \,dx' 
\quad \forall w \in W^{1,2}_{\Gamma_1 \cup \Gamma_L} (\Omega).
\end{array}
\end{eqnarray*}
By using the same technique as in Proposition \ref{prop3} we obtain
\begin{eqnarray*}
\begin{array}{ll}
\displaystyle k_0 \zeta \int_{\Omega} \frac{\bigl| \nabla (\theta_{m_1} - \theta_{m_2} ) \bigr|^2}{ \bigl( 1 + | (\theta_{m_1} - \theta_{m_2} | \bigr)^{\zeta +1} } \, dx 
\le  \Vert (\upsilon_{m_2} - \upsilon_{m_1} ) \cdot\nabla \theta_{m_2}  \Vert_{L^1 (\Omega)}
\\
\displaystyle
+  \Vert g_{1/m_1} ( \theta_{m_1}, \upsilon_{m_1}) - g_{1/m_2} ( \theta_{m_2}, \upsilon_{m_2}) \bigr) \Vert_{L^1(\Omega)}  
+ \Vert r(\theta_{m_1} ) - r(\theta_{m_2} ) \Vert_{L^1(\Omega)}
+ \Vert \theta^{b}_{1/m_1} - \theta^{b}_{1/m_2} \Vert_{L^1(\Gamma_0)}
\end{array}
\end{eqnarray*}
and recalling that $\displaystyle \zeta  \in \left( 0, \frac{3 -2q}{3-q} \right)$ and  $\displaystyle \rho = \frac{(\zeta+1) q}{2-q} \in (1, q_*)$ we get
\begin{eqnarray} \label{cauchy1}
\Vert \theta_{m_1} - \theta_{m_2}  \Vert_{1.q} \le C_{m_1, m_2} 2^{(\rho-1) \frac{2-q}{2q}} \Bigl( |\Omega|^{\frac{2-q}{2q}} + \Vert \theta_{m_1} - \theta_{m_2}  \Vert_{L^{\rho}(\Omega)}^{\rho \frac{2-q}{2q}} \Bigr) 
\end{eqnarray}
with
\begin{eqnarray} \label{cauchy2}
\begin{array}{ll}
\displaystyle C_{m_1, m_2} =  \frac{1}{ \sqrt{k_0 \zeta}} 
\Bigl( \Vert (\upsilon_{m_2} - \upsilon_{m_1} ) \cdot\nabla \theta_{m_2}  \Vert_{L^1 (\Omega)}
+  \Vert g_{1/m_1} ( \theta_{m_1}, \upsilon_{m_1}) - g_{1/m_2} ( \theta_{m_2}, \upsilon_{m_2}) \bigr) \Vert_{L^1(\Omega)}   \\
\displaystyle 
\qquad \qquad  + \Vert r(\theta_{m_1} ) - r(\theta_{m_2} ) \Vert_{L^1(\Omega)}
+ \Vert \theta^{b}_{1/m_1} - \theta^{b}_{1/m_2} \Vert_{L^1(\Gamma_0)} \Bigr)^{1/2}.
\end{array}
\end{eqnarray}

If $\displaystyle q \in \bigl( 1, 3/2 \bigr)$ we have
\begin{eqnarray*}
\Vert (\upsilon_{m_2} - \upsilon_{m_1} ) \cdot\nabla \theta_{m_2}  \Vert_{L^1 (\Omega)} 
\le \Vert \upsilon_{m_2} - \upsilon_{m_1} \Vert_{{\bf L}^{q'}(\Omega)} \Vert \nabla \theta_{m_2} \Vert_{L^q(\Omega)}
= \Vert \upsilon_{m_2} - \upsilon_{m_1} \Vert_{{\bf L}^{q'}(\Omega)} \Vert  \theta_{m_2} \Vert_{1.q}
\end{eqnarray*}
and $V_{0.div}^p$ is continuously embedded into ${\bf L}^{q'}(\Omega)$ thanks to the compatibility condition $\displaystyle p \ge  \frac{3q}{4q -3 }$. Since $(\theta_m)_{m \ge 1}$ is bounded in $W^{1,q}_{\Gamma_1 \cup \Gamma_L} (\Omega)$
we infer from (\ref{conv4}) that
\begin{eqnarray*}
\lim_{m_1, m_2 \to + \infty} \Vert (\upsilon_{m_2} - \upsilon_{m_1} ) \cdot\nabla \theta_{m_2}  \Vert_{L^1 (\Omega)}  =0.
\end{eqnarray*}
If $q=1$ we have $p>3$ and $V_{0.div}^p$ is continuously embedded into ${\bf C}^0 \bigl( {\overline \Omega} \bigr) $ and we have 
\begin{eqnarray*}
\Vert (\upsilon_{m_2} - \upsilon_{m_1} ) \cdot\nabla \theta_{m_2}  \Vert_{L^1 (\Omega)} 
\le \Vert \upsilon_{m_2} - \upsilon_{m_1} \Vert_{{\bf C}^0 ( {\overline \Omega})} \Vert \nabla \theta_{m_2} \Vert_{L^1(\Omega)} 
= \Vert \upsilon_{m_2} - \upsilon_{m_1} \Vert_{{\bf C}^0 ( {\overline \Omega})} \Vert  \theta_{m_2} \Vert_{1.q}
\end{eqnarray*}
and we obtain once again
\begin{eqnarray*}
\lim_{m_1, m_2 \to + \infty} \Vert (\upsilon_{m_2} - \upsilon_{m_1} ) \cdot\nabla \theta_{m_2}  \Vert_{L^1 (\Omega)}  =0.
\end{eqnarray*}
Hence with (\ref{cauchy1})-(\ref{cauchy2}) and (\ref{conv6})-(\ref{conv7})-(\ref{conv8}) we may conclude that the (sub-)sequence $(\theta_m)_{m \ge 1}$ is a Cauchy sequence in $W^{1,q}_{\Gamma_1 \cup \Gamma_L} (\Omega)$ and with (\ref{conv2}) we infer that
\begin{eqnarray*}
\theta_m \to \theta \quad \hbox{\rm strongly in $W^{1,q}_{\Gamma_1 \cup \Gamma_L} (\Omega)$.}
\end{eqnarray*}
Finally we choose $w \in C^1 ({\overline \Omega})$ such that $w=0$ on $\Gamma_1 \cup \Gamma_L$ as a test-function in (\ref{heat_m_bis}). We can now pass to the limit in all the terms and we get 
\begin{eqnarray*}
\begin{array}{ll}
\displaystyle  \int _{\Omega} \bigl( (\upsilon+G) \cdot\nabla\theta \bigr)w\,dx
+ \int _{\Omega} (K \nabla \theta) \cdot \nabla w\,dx 
= \int _{\Omega}2 \mu \bigl(\theta , \upsilon +G, \bigl|D(\upsilon +G) \bigr| \bigr)
\bigl|D(\upsilon +G) \bigr|^p w \,dx 
 \\
\displaystyle 
+\int _{\Omega} r(\theta ) w \,dx + \int_{\Gamma_{0}} \theta^{b} w \,dx' \quad \forall w \in {\mathcal D} ({\overline  \Omega}) \ {\rm s.t.}  \ w = 0 \ {\rm on} \ \Gamma_1 \cup \Gamma_L.
\end{array}
\end{eqnarray*}
Together with Proposition \ref{prop5} we infer that the triplet $(\upsilon, \pi, \theta)$ is a solution to problem (P).

\end{proof}

%%%%%%%%%%%%%%%%%%%%%%%%%%%%%%%%%%%%%%%%%%%%%%%%%%%%%%%%%%%%%%%%%
%                Bibliographie
%%%%%%%%%%%%%%%%%%%%%%%%%%%%%%%%%%%%%%%%%%%%%%%%%%%%%%%%%%%%%%%%%


\begin{thebibliography}{00}

%% \bibitem{label}
%% Text of bibliographic item

%\bibitem{}

%\bibitem{allaire2}{G. Allaire.}
%\textit{Homogenization and two-scale convergence.} SIAM J. Math. Anal, vol. 23,  1482-1518, 1992.

%\begin{thebibliography}{00} 

%\bibitem{Adams} 
%R.A. Adams, 
%Sobolev Spaces, Academic Press, INC. 1975.

\bibitem{Amrouche-Girault}
C. Amrouche, V. Girault,  
\emph{ Decomposition of vector spaces and application to the Stokes problem in arbitrary dimension},
{Czechoslovak Mathematical Journal}   44(1994)109--140. 

%\bibitem{Ayadi2010} M. Ayadi, M. K. Gdoura and T. Sassi,
%\emph{Mixed formulation for Stokes problem with Tresca friction}, 
%C. R. Acad. Sci. Paris Ser. I, 348 (2010), pp. 1069-1072.

%\bibitem{Ayadi}{M. Ayadi, M.K. Gdoura,  T. Sassi}. 
%\emph{Error estimates for Stokes problem with Tresca friction conditions}, 
%ESAIM: M2AN 48 (2014) 1413-1429.
 
 
%\bibitem{Baiocchi} {C. Baiocchi}, {A. Capelo}. 
%\emph{Variational and quasivariational inequalities applications to free boundary problems}, 
%{ John Wiley  and  Sons,  New York, 1984}.

\bibitem{barnes}
H.A. Barnes, 
\emph{Shear-thickening (``Dilatancy'') in suspensions on nonaggregating solid particles dispersed in Newtonian liquids}, 
{Journal of Rheology} 33/2(1989)329--366.



%\bibitem{Bayada} 
%G. Bayada, M. Boukrouche, 
%\emph{On a free boundary problem for Reynolds equation derived
%from the Stokes system with Tresca boundary conditions},  Journal of Mathematical Analysis
%and Applications, Vol. 282, 212-231,   2003.

%\bibitem{Bensedik}  
%A. Bensedik, M. Boukrouche, 
%\emph{Existence result for a strongly coupled problem with heat convection term and Tresca’s law},  
%Journal of Advanced Research in Differential Equations, Vol. 3, 33-53,  2011.

% \bibitem{P.B} 
% P. B\'enilan, L. Boccardo, T. Gallou\"et, R. Gariepy, M. Pierre, J. L.  V\'azquez,  
%  \emph{An $ L^{1}$-theory of existence and uniqueness of solutions of nonlinear elliptic equations},  
%  Annali della Scuola Normale Superiore di Pisa-Classe di Scienze, 22(2), 241-273, 1995.

\bibitem{berselli}
L.C. Berselli, L. Diening, M. Ru\v zi\v cka, 
\emph{Existence of strong solutions for incompressible fluids with shear dependent viscosities},
 J. Math. Fluid Mech. 12/1(2010)101--132.

\bibitem{Bocca} 
L. Boccardo, T.  Gallou\"et,  
\emph{Nonlinear elliptic and parabolic equations involving measure data}, 
Journal of Functional Analysis  87(1989)149--169.

\bibitem{Bocca2} 
L. Boccardo, T.  Gallou\"et,  
\emph{Nonlinear elliptic equations with right-hand side measures},
 Comm. Partial Differential Equations 17/3-4(1992)641--655. 
 
\bibitem{Boccardo-Neumann-boundary} 
L. Boccardo, L. Moreno-M\'erida, 
\emph{$W^{1,1}(\Omega)$ solutions of Nonlinear Problems with Nonhomogeneous boundary conditions},
Milan J. Math. 83(2015)279--293.

%\bibitem{Boccardo-Dirichlet-boundary}
%L. Boccardo, L. Orsina, 
%\emph{Some borderline cases of nonlinear parabolic equations with irregular data}
%J. Evol. Equ. 16 (2016), 51-64.


%\bibitem{bL2008}
%M.Boukrouche; G. Lukaszewicz, 
%\emph{On the existence of pullback attractor for a two-dimensional shear flow with 
%Tresca's boundary condition.}
%Parabolic and Navier-Stokes equations. Part 1, 81–93, 
%Banach Center Publ., 81, Part 1, Polish Acad. Sci. Inst. Math., Warsaw, 2008.


%\bibitem{Luka2} 
%M. Boukrouche, G.  {\L}ukaszewicz, 
%\emph{On a lubrication problem with Fourier and Tresca boundary conditions},
%Mathematical Models and Methods in Applied Sciences, Vol. 14,  913-941, 2004. 

%\bibitem{MB-RM2008} 
%M. Boukrouche, R. El Mir,  
%\emph{On a non-isothermal, non-Newtonian lubrication problem  with Tresca law:  Existence and the behavior of weak solution},
%Nonlinear Analysis Real Worlds Applications, Vol. 9, 674-692,  2008. 

% \bibitem{imane3}
% M. Boukrouche, I. Boussetouan,  L.  Paoli, 
% \emph{Non-isothermal Navier-Stokes system with mixed boundary conditions and friction law: uniqueness and regularity properties},
% Nonlinear Analysis: Theory, Methods and Applications, Vol. 102, 168-185, 2014.
 
% \bibitem{imane2}
% M. Boukrouche, I. Boussetouan,  L.  Paoli,   
% \emph{Existence for non-isothermal fluid flows with Tresca's friction and Cattaneo's heat law}, 
% Journal of Mathematical Analysis and Applications, Vol. 427, 499-514, 2015.
 
% \bibitem{imane4}
% M. Boukrouche, I. Boussetouan,  L.  Paoli, 
% \emph{Unsteady 3D-Navier-Stokes system with Tresca’s friction law},
% arXiv preprint arXiv:1512.06607, 2015. 

%\bibitem{Saidi}
%M. Boukrouche, F. Saidi, 
%\emph{ Non-isothermal lubrication problem with
%Tresca fluid-solid interface  law.  Part 1}. 
%Nonlinear Anal. Real World Applications.
%Volume 7, Issue 5, December 2006, Pages 1145-1166.
 

 
% \bibitem{malek2} M. Bul\'icek, M. Majdoub, and J. M\'alek,
% \emph{Unsteady flows of fluids with pressure dependent viscosity in
%unbounded domains.} 
%Nonlinear Analysis: Real World Applications 11 (2010)
%3968-3983

%\bibitem{malek1} 
%M. Bul\'icek, J. M\'alek,  and K. R. Rajagopal.
% \emph{ Analysis of the flows of incompressible fluids with pressure dependent
%viscosity fulffilling $\nu(p , \cdot) \to +\infty$ as $p\to +\infty$.}
%Czechoslovak Mathematical Journal, 59 (134) (2009), 503-528.

%\bibitem{gwia}  
%M. Bul\'icek, J. M\'alek,  and A. \'Swierczewska-Gwiazda, 
% \emph{ On steady flows of incompressible fluids with implicit power-law-like rheology.}
%Advances in calculus of variations 2 (2009), no. 2, 109-136.


% \bibitem{Brezis99} 
% H. Brezis,  
% \emph{Analyse fonctionnelle théorie et applications},
% Dunod, {Paris}, 1999.
 





%\bibitem{S.S}  
%S.Segura  De Le{\'o}n, 
%{ Existence and uniqueness for $ L^1$ data of some elliptic equations with natural growth},
%Advances in Differential Equations, Vol. 8,  1377-1408,  2003.

\bibitem{diening}
L. Diening, M. Ru\v zi\v cka, J. Wolf,
\emph{Existence of weak solutions for unsteady motions of generalized Newtonian fluids}
 Ann. Sc. Norm. Super. Pisa Cl. Sci. 5(9)/1(2010)1--46. 

\bibitem{Duvaut-Lions72} 
G. Duvaut, J.L. Lions, 
\emph{Les in\'equations en m\'ecanique et physique},
Dunod, Paris, 1972.



%\bibitem{ge}  
%M. Fang, R.P. Gilbert, 
%{ Nonlinear systems arising from non-isothermal, non-Newtonian Hele-Shaw flows
%in the presence of body forces and sources. }
%Mathematical and Computer Modelling, 35 (2002) 1425-1444.



%  \bibitem{Farhloul2017}
%M. Farhloul,  A.Zine, 
%\emph{A dual-mixed finite element method for quasi-Newtonian flows whose viscosity obeys a power law or the Carreau law.}
%Math.Comput. Simulation 141 (2017), 83–95.

\bibitem{fujita1}
 H. Fujita,
\emph{Flow Problems with Unilateral Boundary Conditions.}
 Le\c cons, Coll\`ege de France (1993).

 \bibitem{F1} 
 H. Fujita, 
  \emph{A mathematical analysis of motions of viscous incompressible fluid under leak or slip boundary conditions}, 
  Mathematical Fluid Mechanics and Modeling  888(1994)199--216.
  
 \bibitem{F2} 
 H. Fujita, H. Kawarada,  A. Sasamoto, 
 \emph {Analytical and numerical approaches to stationary flow problems with leak and slip boundary conditions},
Advances in Numerical Mathematics, Lecture Notes Numer. Appl. Anal. 14(1995)17--31.






 \bibitem{F3} 
 H. Fujita, H.   Kawarada, 
 \emph{Variational inequalities for the Stokes equation with boundary conditions of friction type},  
 Recent Developments in Domain Decomposition Methods and Flow Problems, GAKUTO Internat. Ser. Math. Sci. Appl.,   11(1998)15--33.
 
 
 



% \bibitem{F4}
%  H. Fujita,  
% \emph{Non-stationary Stokes flows under leak boundary conditions of friction type}, 
% Journal of Computational and Applied Mathematics,  Vol. 19,  1–8, 2001.
 
 
  \bibitem{F6} 
  H. Fujita,  
  \emph{A coherent analysis of Stokes flows under boundary conditions of friction type}, 
  Journal of Computational and Applied Mathematics  149(2002)57--69.
 
 

%\bibitem{Andrey2016}
%A.A. Gavrilov, V.Ya Rudyak, 
%\emph{Reynolds-averaged modeling of turbulent flows of power-law fluids},
%Journal of  Non-Newtonian Fluids Mechanics Vol.227, January 2016 45-55. 
  
  

\bibitem{Hervet2003} 
H. Hervet, L. L\'eger, 
\emph{Flow with slip at the wall: from simple to complex fluids},
 C.R.Acad.Sci. Paris Physique
4(2003)241--249.



%\bibitem{Bin-Hu}
%B. Hu, S. Kieweg, 
%\emph{ Line instability of gravity-driven flow of power-law fluids}, 
%Journal of  Non-Newtonian Fluids Mechanics Vol.225, November 2015, 62-69.

\bibitem{ladi}
O.A. Ladyzhenskaya,
\emph{New equations for the description of the motions of viscous incompressible fluids, and global solvability for their boundary value problems},
 Boundary value problems of mathematical physics Part 5, Trudy Mat. Inst. Steklov 102(1967)85--104.



\bibitem{LeDret2013} 
H. Le Dret,  
\emph{Equations aux d\'eriv\'ees partielles non lin\'eaires}, 
Math\'ematiques et Applications vol. 72,    Springer-Verlag, Berlin,  2013.



\bibitem{Roux2} 
C. Le Roux, A. Tani,  
\emph{Steady flows of incompressible Newtonian fluids with threshold slip  boundary conditions},
Mathematical Analysis in Fluid and Gas Dynamics  1353(2004)21--34.


 \bibitem{Roux3}
 C. Le  Roux,   A. Tani, 
 \emph{Steady solutions of the Navier–Stokes equations with threshold slip boundary conditions}, 
 Mathematical Methods in the Applied Sciences 30(2007)595--624.

 
% \bibitem{lions78} 
%J.~L. Lions,
%\emph{Some problems connected with Navier-Stokes equations}, in Lectures at the $IV^{TH}$ 
%Latin-American School of Mathematics, Lima, July 1978.
 
 
 
 \bibitem{Lions69} 
 J.L. Lions, 
 \emph{Quelques m\'ethodes de r\'esolution des probl\`emes aux limites non lin\'eaires}, 
 Dunod, Paris, 1969.
 
%\bibitem{majda1984}  A. Majda, 
%Compressible Fluid flow and Systems of conservation laws in several  space variables. 
%Applied Mathematical Sciences 53, Springer-Verlag, 1984.

\bibitem{magnin}
A. Magnin,  J.M. Piau,
\emph{Shear rheometry of fluids with a yield stress},
 J. Non-Newtonian Fluid Mech. 23(1987)91--106.


% \bibitem{Murat} 
% F. Murat,
% \emph{ \'Equations elliptiques non linéaires avec second membre $L^1$  ou mesure},
% {26 \`eme Congr\'es National d’Analyse num\'erique. Les Karellis, 1994.}


%\bibitem{tapiero} 
%A. Mikeli\'c,  R. Tapi\'ero, 
%\emph{ Mathematical derivation of the power law describing polymer flow through a thin slab. } 
%RAIRO Mod. Math. Anal. 29(1) (1995) 3-21.

\bibitem{malek}
J. Malek, K.R. Rajagopal,
\emph{ Mathematical issues concerning the Navier-Stokes equations
and some of its generalizations},
 Evolutionary equations, Handb. Differ. Equ., 
 Elsevier/North-Holland, 
2(2005)371--459.

\bibitem{mewis}
J. Mewis, N.J. Wagner, 
\emph{Colloidal Suspension Rheology}, 
Cambridge University Press, Cambridge, 2012.



\bibitem{mosolov}
P.P. Mosolov, V.P. Mjasnikov,
\emph{A proof of Korn's inequality},
Dokl. Akad. Nauk. SSSR, Soviet Math. Dokl., 12/6(1972)1618--1622.

\bibitem{oden}
J.T. Oden,
\emph{Qualitative methods in non linear mechanics},
Prentice Hall, New-York, 1986.


\bibitem{rock}  
R.T. Rockafellar,   
\emph{Convex analysis},  
Princeton University Press, Princeton, 1970.

%\bibitem{Rao} I.J. Rao, K.R. Rajagopal,  
%\emph{The effect of the slip boundary condition on the flow of fluids in a channel}, 
%Acta Mechanica, Vol. 135, 113–126, 1999.

%\bibitem{RodrigesJR2015} J.F. Rodrigues,
%\emph{On the mathematiacal analysis of thick fluids,}
%Journal of Mathematical Sciences, Vol. 210, No. 6, November, 2015.

\bibitem{saito1}
 N. Saito, H. Fujita,
\emph{Regularity of solutions to the Stokes equation under a certain nonlinear boundary condition},
 The Navier-Stokes equations, Lecture
Note Pure Appl. Math. 223(2001)73--86.


\bibitem{saito2}  
 N. Saito,
\emph{On the Stokes equations with the leak and slip boundary conditions of friction type: regularity of solutions},
 Publ. RIMS Kyoto Univ. 40(2004)345--383.

%\bibitem{sandri} D. Sandri,
%\emph{Numerical analysis of a four-field model for the approximation of a fluid
%obeying the power law or Carreau's law}.
%Jpn. J. Ind. Appl. Math. 31 (2014), no. 3, 633–663.

\bibitem{wagner}
N.J. Wagner, J.F. Brady, 
\emph{Shear thickening in colloidal dispersions}, Physics Today, 62/10(2009)27--32.


\end{thebibliography}
\end{document}